\documentclass[a4paper]{amsart}
\usepackage[foot]{amsaddr}
\usepackage[utf8]{inputenc}
\usepackage[T1]{fontenc}
\usepackage{lmodern}
\usepackage{amsmath}
\usepackage{amsfonts}
\usepackage{amssymb}
\usepackage{amsthm}
\usepackage{thmtools}
\usepackage{mathtools}
\usepackage{tikz-cd}
\usepackage[inline]{enumitem}
\usepackage[pdftex,
            pdfauthor={Sven Möller},
            pdftitle={Natural Construction of Ten Borcherds-Kac-Moody Algebras Associated with Elements in M23},
            pdfsubject={},
            pdfkeywords={},
            pdfproducer={},
            pdfcreator={}]{hyperref}

\theoremstyle{plain}
\newtheorem{thm}{Theorem}[section]
\newtheorem{prop}[thm]{Proposition}
\newtheorem{cor}[thm]{Corollary}
\newtheorem{lem}[thm]{Lemma}

\newtheorem*{thm*}{Theorem}

\theoremstyle{definition}
\newtheorem{defi}[thm]{Definition}
\newtheorem*{defi*}{Definition}

\newtheorem*{nota*}{Notation}
\newtheorem{rem}[thm]{Remark}

\newcommand{\Q}{\mathbb{Q}}
\newcommand{\Z}{\mathbb{Z}}
\newcommand{\Ns}{\mathbb{Z}_{>0}}
\newcommand{\N}{\mathbb{Z}_{\geq0}}
\newcommand{\C}{\mathbb{C}}
\newcommand{\R}{\mathbb{R}}
\renewcommand{\H}{\mathbb{H}}

\newcommand{\tr}{\operatorname{tr}}
\renewcommand{\i}{\mathrm{i}}
\newcommand{\e}{\mathrm{e}}
\newcommand{\End}{\operatorname{End}}
\newcommand{\Aut}{\operatorname{Aut}}

\newcommand{\rk}{\operatorname{rk}}
\newcommand{\im}{\operatorname{im}}

\newcommand{\voa}{vertex operator algebra}

\newcommand{\vosa}{vertex operator subalgebra}

\newcommand{\aia}{abelian intertwining algebra}

\newcommand{\fpvosa}{fixed-point vertex operator subalgebra}

\newcommand{\BKMa}{Borcherds-Kac-Moody algebra}
\newcommand{\BKMA}{Borcherds-Kac-Moody Algebra}
\newcommand{\fqs}{finite quadratic space}

\newcommand{\vac}{\textbf{1}}
\newcommand{\ch}{\operatorname{ch}}

\newcommand{\id}{\operatorname{id}}

\newcommand{\eps}{\varepsilon}

\newcommand{\SLZ}{\operatorname{SL}_2(\mathbb{Z})}
\newcommand{\GL}{\operatorname{GL}}
\newcommand{\MpZ}{\operatorname{Mp}_2(\mathbb{Z})}
\newcommand{\ee}{\mathfrak{e}}
\newcommand{\g}{\mathfrak{g}}
\newcommand{\h}{\mathfrak{h}}

\newcommand{\strathol}{strongly rational, holomorphic}
\newcommand{\strat}{strongly rational}
\newcommand{\II}{I\!I}
\renewcommand{\O}{\operatorname{O}}
\newcommand{\oddity}{\operatorname{oddity}}
\renewcommand{\Im}{\operatorname{Im}}

\newlength{\myl}
\begin{document}

\title[Natural Construction of Ten Borcherds-Kac-Moody Algebras]{Natural Construction of Ten Borcherds-Kac-Moody Algebras Associated with Elements in $M_{23}$}
\author[Sven Möller]{Sven Möller}
\address{Rutgers University, Piscataway, NJ, United States of America}
\email{\href{mailto:math@moeller-sven.de}{math@moeller-sven.de}}

\begin{abstract}
\BKMa{}s generalise finite-dimensional, simple Lie algebras. Scheithauer showed that there are exactly ten \BKMa{}s whose denominator identities are completely reflective automorphic products of singular weight on lattices of square-free level. These belong to a larger class of Borcherds-Kac-Moody (super)algebras Borcherds obtained by twisting the denominator identity of the Fake Monster Lie algebra. Borcherds asked whether these Lie (super)algebras admit natural constructions. For the ten Lie algebras from the classification we give a positive answer to this question, i.e.\ we prove that they can be realised uniformly as the BRST cohomology of suitable vertex algebras.
\end{abstract}

\vspace*{-25pt}
\maketitle

\vspace*{-25pt}
\setcounter{tocdepth}{1}
\tableofcontents
\setcounter{tocdepth}{2}

\vspace*{-25pt}


\section{Introduction}
There is an intriguing relation between vertex algebras, \BKMa{}s and automorphic forms. Vertex (operator) algebras give a mathematically rigorous description of two-dimensional conformal field theories \cite{Bor86,FLM88}. \BKMa{}s (or generalised Kac-Moody algebras) are natural generalisations of finite-dimensional, simple Lie algebras defined by generators and relations \cite{Bor88}. Both concepts were famously used by Borcherds in his proof of the Monstrous Moonshine conjecture \cite{Bor92}.

Sometimes, the denominator identities of \BKMa{}s are automorphic forms on orthogonal groups. Classification results for such Lie algebras were obtained in \cite{GN02,Sch06,GN18}. It is conjectured that those \BKMa{}s whose denominator identities are automorphic products of singular weight \cite{Bor98} can all be realised in natural constructions, i.e.\ other than by generators and relations, for example as string quantisations of suitable conformal vertex algebras $M$ of central charge $26$ (see Problem~8 in \cite{Bor01}).

On the other hand, in \cite{Bor92} a large class of Borcherds-Kac-Moody (super)al\-ge\-bras $\g_{\phi_\nu}$ was obtained by twisting the denominator identity of the Fake Monster Lie algebra \cite{Bor90} by elements $\nu\in\O(\Lambda)\cong\mathrm{Co}_0$, the automorphism group of the Leech lattice~$\Lambda$, or, more precisely, by their standard lifts $\phi_\nu$ (see Section~\ref{sec:twisting}). Borcherds then asked if also these Lie algebras can be obtained in natural constructions (see \cite{Bor92}, Section~15).

In Section~\ref{sec:brst}, as our main result, we give bosonic string constructions (based on the BRST or semi-infinite cohomology \cite{Fei84,FGZ86}) of ten particularly nice special cases, namely those \BKMa{}s $\g_{\phi_\nu}$ associated with the elements $\nu$ of square-free order $m$ in the Mathieu group $M_{23}$, viewed as a subgroup of $\O(\Lambda)$, hence giving a partial positive answer to Borcherds' question.
\begin{thm*}[Main Result, Theorem~\ref{thm:main}]
Let $\nu$ be of square-free order in $M_{23}$. Then there is a conformal vertex algebra $M_{\phi_\nu}$ of central charge $26$ whose BRST cohomology $H^1_\mathrm{BRST}(M_{\phi_\nu})$ is isomorphic to the (complexification of the) \BKMa{} $\g_{\phi_\nu}$ obtained by twisting the denominator identity of the Fake Monster Lie algebra by $\phi_\nu$.
\end{thm*}
This also adds evidence to the aforementioned conjecture since it was proved in \cite{Sch06} that these ten \BKMa{}s are exactly those whose denominator identities are completely reflective automorphic products of singular weight on lattices of square-free level (see Section~\ref{sec:class}).

In fact, we shall see that their denominator identities are Borcherds lifts of certain vector-valued characters associated with the vertex algebras $M_{\phi_\nu}$ in the input of the BRST construction (see Proposition~\ref{prop:F1eqF2} and Remark~\ref{rem:vvmfF}).

The construction of $M_{\phi_\nu}$ as a certain simple-current extension involving the \fpvosa{} $V_\Lambda^{\phi_\nu}$ of the Leech lattice \voa{} $V_\Lambda$ (see Proposition~\ref{prop:mdef}) is made possible by recent advancements in orbifold theory, most notably in \cite{Lam19} and \cite{EMS20a,Moe16} (see Section~\ref{sec:orb}).

Some of these ten \BKMa{}s have already been constructed as string quantisations. Clearly, for $\nu=\id$ we obtain the Fake Monster Lie algebra \cite{Bor90} itself. For the automorphism of order $2$ in $M_{23}$ one obtains the Fake Baby Monster Lie algebra \cite{HS03}. With a slightly less effective method in \cite{CKS07} the authors constructed the four \BKMa{}s associated with the automorphisms in $M_{23}$ of order $2$, $3$, $5$ and $7$ depending on some conjectures.

The main notions and their connections are depicted in the following diagram and in the diagram at the end of Section~\ref{sec:brst}:
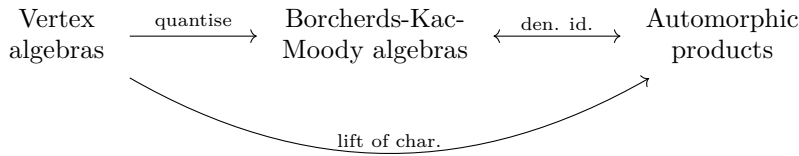
\begin{equation*}
\begin{tikzcd}
\begin{tabular}{c}Vertex\\algebras\end{tabular}\arrow{rr}{\text{quantise}}\arrow[bend right]{rrrr}{\text{lift of char.}}&&\begin{tabular}{c}Borcherds-Kac-\\Moody algebras\end{tabular}\arrow[<->]{rr}{\text{den. id.}}&&\begin{tabular}{c}Automorphic\\products\end{tabular}
\end{tikzcd}
\end{equation*}

The paper is organised as follows:

In Section~\ref{sec:bkma} we state a sufficient criterion for a Lie algebra to be a \BKMa{}, describe Borcherds' twisting procedure for the Fake Monster Lie algebra and state Scheithauer's classification result.

In Section~\ref{sec:va} we describe orbifold results for \voa{}s associated with coinvariant sublattices of unimodular lattices and then construct the ten conformal vertex algebras of central charge $26$ that will serve as input for the BRST quantisation construction.

In Section~\ref{sec:brst} we describe the BRST quantisation, study the ten \BKMa{}s obtained in this procedure and state the main result of the paper (Theorem~\ref{thm:main}).

\subsubsection*{Conventions}
All Lie algebras and vertex algebras will be over the base field $\C$ unless otherwise noted, in which case they will be over $\R$. Note that $\tau$ will always be assumed to be in the upper half-plane $\H=\{z\in\C\,|\,\Im(z)>0\}$ and $q=\e^{(2\pi\i)\tau}$.

\subsection*{Acknowledgements}
The author would like to thank Jethro van Ekeren, Gerald Höhn, Ching Hung Lam, Jim Lepowsky, Maximilian Rössler and Nils Scheithauer for helpful discussions. The author thanks the two anonymous referees for their detailed comments. The author was partially supported by a scholarship from the \emph{Studienstiftung des deutschen Volkes} and by the \emph{Deutsche Forschungsgemeinschaft} through the project ``Infinite-dimensional Lie Algebras in String Theory''.


\section{\BKMA{}s}\label{sec:bkma}
In this section we discuss \BKMa{}s and in particular the ten \BKMa{}s for which we develop string constructions in this text.

\emph{\BKMa{}s} (or \emph{generalised Kac-Moody algebras}) are a class of infinite-dimensional Lie algebras introduced in \cite{Bor88} (see also \cite{Bor92} and \cite{Jur96}) naturally generalising Kac-Moody algebras, which in turn generalise finite-dimensional, simple Lie algebras. Like Kac-Moody algebras, \BKMa{}s are defined by generators and relations, which are encoded in a generalised Cartan matrix. However, the restrictions on the generalised Cartan matrix are weaker, and, in particular, simple roots may be imaginary.

\BKMa{}s admit representation-theoretic data like a character formula for highest-weight modules and a denominator identity
\begin{equation*}
\ee^\rho\prod_{\alpha\in\Phi^+}(1-\ee^\alpha)^{\operatorname{mult}(\alpha)}=\sum_{w\in W}\det(w)w\left(\ee^\rho\sum_{\alpha\in\Phi}\eps(\alpha)\ee^\alpha\right),
\end{equation*}
an identity of formal exponentials, where the second sum is over all roots $\alpha$ in the root system $\Phi$ and the product ranges over the set $\Phi^+$ of positive roots, $W$ denotes the Weyl group, $\rho$ the Weyl vector, $\operatorname{mult}(\alpha)$ the multiplicity of the root $\alpha$ and $\eps(\alpha)$ is $(-1)^n$ if $\alpha$ is the sum of $n$ pairwise orthogonal imaginary simple roots and 0 otherwise.


\subsection{Borcherds-Kac-Moody Property}
In the following we state a sufficient criterion that will allow us to identify complex Lie algebras as \BKMa{}s. It is a slight modification of Theorem~1 in \cite{Bor95a} where the case of real Lie algebras was treated:
\begin{prop}[\cite{Car12b}, Lemma~3.4.2]\label{prop:bkma}
Let $\g$ be a complex Lie algebra satisfying the following conditions:
\begin{enumerate}
\item\label{enum:bkma1} $\g$ admits a non-degenerate, symmetric, invariant bilinear form $(\cdot,\cdot)$.
\item\label{enum:bkma2} $\g$ has a self-centralising subalgebra $\mathcal{H}$, called a \emph{Cartan subalgebra}, such that $\g$ is the direct sum of eigenspaces under the adjoint action of $\mathcal{H}$ and the non-zero eigenvalues, called \emph{roots}, have finite multiplicity.
\item\label{enum:bkma3} There is a real subspace $\mathcal{H}_\R$ of $\mathcal{H}$ such that $\mathcal{H}=\mathcal{H}_\R\otimes_\R\C$, the bilinear form is real-valued on $\mathcal{H}_\R$ and the roots lie in the dual space $(\mathcal{H}_\R)^*$.
\item\label{enum:bkma4} The norms of the roots under the inner product $(\cdot,\cdot)$ are bounded from above.
\item\label{enum:bkma5} There exists a vector $h_\mathrm{reg.}\in\mathcal{H}_\R$, called a \emph{regular element}, such that:
\begin{enumerate}
\item $\mathcal{H}=C_\g(h_\mathrm{reg.})$, the centraliser of $h_\mathrm{reg.}$ in $\g$,
\item for any $M\in \R$, there exist only finitely many roots $\alpha$ such that $|\alpha(h_\mathrm{reg.})|<M$.
\end{enumerate}
(If $\alpha(h_\mathrm{reg.})<0$, then we say that the root $\alpha$ is \emph{negative} and if $\alpha(h_\mathrm{reg.})>0$, we say that $\alpha$ is \emph{positive}.)
\item\label{enum:bkma6} Any two roots of non-positive norm that are both positive or both negative have inner product at most zero, and if the inner product is zero, then their root spaces commute.
\end{enumerate}
Then $\g$ is a \BKMa{}.
\end{prop}

We can simplify the above criterion if the Lie algebra is Lorentzian, i.e.\ if the bilinear form restricted to $\mathcal{H}_\R$ has Lorentzian signature:
\begin{prop}[cf.\ \cite{Bor95a}, Theorem~2]\label{prop:bor95thm2}
Let $\g$ be a complex Lie algebra satisfying conditions \eqref{enum:bkma1} to \eqref{enum:bkma4} in the above proposition. Assume that the bilinear form restricted to $\mathcal{H}_\R$ is Lorentzian, i.e.\ has signature $(\dim(\mathcal{H})-1,1)$. Then \eqref{enum:bkma5} is fulfilled. Moreover, \eqref{enum:bkma6} is true if additionally the following holds: if two roots are positive multiples of the same norm-zero vector, then their root spaces commute.
\end{prop}
\begin{proof}
This is essentially Theorem~2 in \cite{Bor95a} adapted to the case of complex Lie algebras. In the case of complex $\g$ we have to replace $\mathcal{H}$ by $\mathcal{H}_\R$ and can apply the same arguments.
\end{proof}


\subsection{Twisting the Fake Monster Lie Algebra}\label{sec:twisting}
In \cite{Bor92}, in addition to his famous proof of the Monstrous Moonshine conjecture, Borcherds also constructed a class of Borcherds-Kac-Moody (super)algebras by twisting the denominator identity of the Fake Monster Lie algebra, both as Lie (super)algebras over $\R$. We shall describe this construction and a nice special case in the following.

The \emph{Fake Monster Lie algebra} $\g$ \cite{Bor90}, originally called Monster Lie algebra\footnote{The term \emph{Monster Lie algebra} was later recoined to denote the \BKMa{} obtained as quantisation of $V^\natural\otimes V_{\II_{1,1}}$ (with the Moonshine module $V^\natural$ \cite{FLM88}), which was used by Borcherds in his proof of the Monstrous Moonshine conjecture \cite{Bor92}.} by Borcherds, is the $\II_{25,1}$-graded (real) \BKMa{} obtained as quantisation (see Section~\ref{sec:brst}) of the conformal vertex algebra $M=V_{\II_{25,1}}$ of central charge $26$ associated with the unique even, unimodular lattice $\II_{25,1}$ of Lorentzian signature $(25,1)$.

Let $\Lambda$ denote the Leech lattice, i.e.\ the unique positive-definite, even, unimodular lattice of dimension $24$ that has no roots. The root lattice of the Fake Monster Lie algebra $\g$ is $\II_{25,1}\cong\Lambda\oplus\II_{1,1}$ with elements $\alpha=(\lambda,m,n)$ for $\lambda\in\Lambda$, $m,n\in\Z$ and norm $\langle\alpha,\alpha\rangle/2=\langle\lambda,\lambda\rangle/2-mn$. A non-zero vector $\alpha\in\II_{25,1}$ is a root if and only if $\langle\alpha,\alpha\rangle/2\leq1$, in which case it has multiplicity
\begin{equation*}
\dim(\g(\alpha))=\left[\frac{1}{\eta^{24}}\right]\left(-\frac{\langle\alpha,\alpha\rangle}{2}\right)
\end{equation*}
where $\eta$ is the Dedekind eta function, a modular form of weight $1/2$. The real simple roots of $\g$ are the vectors $\alpha\in\II_{25,1}$ of norm $\langle\alpha,\alpha\rangle/2=1$ satisfying $\langle\alpha,\rho\rangle=-\langle\alpha,\alpha\rangle/2=-1$ where $\rho=(0,0,1)$ is (a choice of) the Weyl vector. They generate the Weyl group $W$ of $\g$, which is the full reflection group of $\II_{25,1}$. The imaginary simple roots are the positive multiples $n\rho$, $n\in\Ns$, of the Weyl vector, each with multiplicity $24$. The denominator identity of $\g$ is
\begin{equation*}
\ee^\rho\prod_{\alpha\in\Phi^+}(1-\ee^\alpha)^{[1/\eta^{24}](-\langle\alpha,\alpha\rangle/2)}=\sum_{w\in W}\det(w)w(\eta^{24}(\ee^\rho)).
\end{equation*}
Note that $\eta^{24}(\ee^\rho)=\ee^\rho\prod_{n=1}^\infty(1-\ee^{n\rho})^{24}$ and recall that $\Phi^+$ denotes the set of positive roots. Upon replacing the formal exponentials by complex ones, the above is the expansion of a certain automorphic product $\Psi$ of weight $12$ for $\O(\II_{26,2})^+$.\footnote{Here, $\O(\II_{26,2})^+$ denotes the subgroup of $\O(\II_{26,2})$ of elements preserving the (choice of continuously varying) orientation on the $2$-dimensional positive-definite subspaces of $\II_{26,2}\otimes_\Z\R$. See, for example, Section~13 in \cite{Bor98}.}

We describe certain automorphisms of the Fake Monster Lie algebra $\g$. The automorphism group of the Leech lattice \voa{} $V_\Lambda$ acts on $M=V_{\II_{25,1}}\cong V_\Lambda\otimes V_{\II_{1,1}}$ by trivially extending the automorphisms to the tensor product. This implies that $\Aut(V_\Lambda)$ acts as Lie algebra automorphisms on $\g$ (see comment after Proposition~\ref{prop:valie}). Explicitly, by the vanishing theorem (see Propositions \ref{prop:vanish2} and \ref{prop:vanish2_0}), viewing $M$ and $\g$ only as $\II_{1,1}$-graded for the moment, the graded component $\g(\beta)$ is isomorphic as an $\Aut(V_\Lambda)$-module to the $L_0$-weight space $(V_\Lambda)_{1-\langle\beta,\beta\rangle/2}$ for non-zero $\beta\in\II_{1,1}$ and to $(V_\Lambda)_1\oplus\R^{1,1}$ for $\beta=0$.

Given an automorphism of the \BKMa{} $\g$, Borcherds defined a \emph{twisted denominator identity} \cite{Bor92}. Sometimes, this will be the (untwisted) denominator identity of some other Borcherds-Kac-Moody (super)algebra. We describe some special cases. For an automorphism $\nu\in\O(\Lambda)$ of order $m$, let $\phi_\nu\in\O(\hat{\Lambda})\leq\Aut(V_\Lambda)$ be the (up to conjugacy unique) standard lift of $\nu$ (see Section~\ref{sec:lattice}). For simplicity, we also assume that $\phi_\nu^k$ is a standard lift of $\nu^k$ for all $k\in\N$, which is for example the case if $\nu$ has odd order. In particular, $\phi_\nu$ has the same order as $\nu$. Borcherds then computes the corresponding twisted denominator identity and shows that it is the denominator identity of a real Borcherds-Kac-Moody superalgebra, which we shall call $\g_{\phi_\nu}$ in the following. Depending on $\nu$, this Lie superalgebra $\g_{\phi_\nu}$ will sometimes be a Lie algebra.

For a lattice automorphism $\nu$ of order $m$ with cycle shape $\prod_{t\mid m}t^{b_t}$, $b_t\in\Z$, we define the associated eta product as $\eta_\nu(\tau):=\prod_{t\mid m}\eta(t\tau)^{b_t}$. The \emph{level} of such an automorphism is defined as the level of the subgroup of $\SLZ$ fixing the eta product $\eta_\nu$ under modular transformations, and this is the smallest positive multiple $N$ of $m=|\nu|$ such that $24$ divides $N\sum_{t\mid m}b_t/t$.

Scheithauer showed that if $\nu$ has square-free level, then the $\phi_\nu$-twisted denominator identity of $\g$, i.e.\ the denominator identity of $\g_{\phi_\nu}$, is an automorphic form of singular weight $-w:=\rk(\Lambda^\nu)/2=:k/2-1$ in the image of the Borcherds lift, i.e.\ an automorphic product (see \cite{Sch06}, Theorem~10.1, \cite{Sch04b,Sch08}). Indeed, starting from the modular form $1/\eta_\nu$ he constructed a vector-valued modular form $F$ of weight $w=1-k/2$ (see Section~\ref{sec:chars}), which he then lifted, using the Borcherds lift \cite{Bor98}, to an automorphic product $\Psi_{\phi_\nu}$ whose expansion at a certain cusp gives the denominator identity of $\g_{\phi_\nu}$.

Finally, we describe the nice special case relevant for this text, which is obtained for ten particular conjugacy classes of automorphisms of the Leech lattice $\Lambda$. Let $m\in\Ns$ be square-free such that $\sigma_1(m)\mid24$ with the sum-of-divisors function $\sigma_1$. Explicitly, let $m=1,2,3,5,6,7,11,14,15,23$. For each such $m$ let $\nu$ be the up to algebraic conjugacy (i.e.\ conjugacy of cyclic subgroups \cite{CCNPW85}) unique\footnote{Except for $m=23$ this is also the unique conjugacy class. When $m=23$, $\nu$ and $\nu^{-1}$ represent two distinct conjugacy classes.} automorphism with cycle shape $\prod_{t\mid m}t^{b_t}=\prod_{t\mid m}t^{24/\sigma_1(m)}$. These automorphisms have order and level $m$. We remark that the fixed-point lattices $\Lambda^\nu$ are the unique even lattices in their respective genera without roots. The rank of $\Lambda^\nu$ is given by $\rk(\Lambda^\nu)=k-2=24\sigma_0(m)/\sigma_1(m)$. The ten automorphisms correspond exactly to the elements of square-free order in the Mathieu group $M_{23}$, which acts naturally on the Leech lattice $\Lambda$, and they are listed in Table~\ref{table:ten} below.

\begin{thm}[\cite{Sch04b}, Theorem~10.1]\label{thm:twistedbkma}
Let $\nu$ be of square-free order in $M_{23}$. Then the $\phi_\nu$-twisted denominator identity of $\g$ is
\begin{equation*}
\ee^\rho\prod_{d\mid m}\prod_{\alpha\in\Phi^+\cap d\Delta'}(1-\ee^\alpha)^{[1/\eta_\nu](-\langle\alpha,\alpha\rangle/2d)}=\sum_{w\in W}\det(w)w(\eta_\nu(\ee^\rho))
\end{equation*}
where $\Delta=\Lambda^\nu\oplus\II_{1,1}$, $\Delta'$ is the dual lattice of $\Delta$, $\rho=(0,0,1)$ and $W$ is the full reflection group of $\Delta$.

This is the denominator identity of the $\Delta$-graded real \BKMa{} $\g_{\phi_\nu}$ whose real simple roots are the simple roots of the Weyl group $W$ and whose imaginary simple roots are the positive multiples $n\rho$, $n\in\Ns$, of the Weyl vector $\rho$ with multiplicity $24\sigma_0((m,n))/\sigma_1(m)$.

This denominator identity is the expansion at any cusp of the automorphic product $\Psi_{\phi_\nu}$ on the lattice $P=L\oplus\II_{1,1}=\Lambda^\nu\oplus\II_{1,1}(m)\oplus\II_{1,1}$ of singular weight $-w=12\sigma_0(m)/\sigma_1(m)\in\Z$ (where $L=\Lambda^\nu\oplus\II_{1,1}(m)$). The lattice $P$ is even, of signature $(k,2)$ and has level $m$.
\end{thm}
We remark that the root multiplicities of $\g_{\phi_\nu}$ are
\begin{equation*}
\dim(\g_{\phi_\nu}(\alpha))=\sum_{d\mid m}\delta_{\alpha\in\Delta\cap d\Delta'}\left[\frac{1}{\eta_\nu}\right]\left(-\frac{1}{d}\frac{\langle\alpha,\alpha\rangle}{2}\right)
\end{equation*}
for all non-zero $\alpha\in\Delta$ and that $\dim(\g_{\phi_\nu}(0))=k$.

One observes that for these ten automorphisms $\nu$ the automorphic form $\Psi_{\phi_\nu}$ is completely reflective (see \cite{Sch06}, Section~9), i.e.\ it has nice symmetries. In fact, as we shall see in the next section, one can show that these are essentially all the completely reflective automorphic products of singular weight on lattices of square-free level.


\subsection{Classification}\label{sec:class}
We describe a classification result for automorphic products and for \BKMa{}s from \cite{Sch06}.

As we saw above, the denominator identity of a \BKMa{} is sometimes an \emph{automorphic product}. These are automorphic forms on orthogonal groups in the image of the \emph{Borcherds lift} \cite{Bor98}, which lifts from vector-valued modular forms for the Weil representation of $\MpZ$. Since these automorphic forms have an infinite-product expansion, they are called automorphic products.

In \cite{Sch06} the author classified all \BKMa{}s whose denominator identities are completely reflective automorphic products of singular weight. He found that the ten \BKMa{}s from Section~\ref{sec:twisting} are essentially all such \BKMa{}s. More precisely:
\begin{thm}[\cite{Sch06}, Theorem~12.7]\label{thm:12.7}
Let $P$ be an even lattice of signature $(k,2)$ with $k\geq 4$, square-free level $m$ and $p$-ranks of the discriminant form $P'/P$ at most $k+1$. Then a real \BKMa{} whose denominator identity is a completely reflective automorphic product of singular weight $-w=k/2-1$ on $P$ is isomorphic to $\g_{\phi_\nu}$ for the automorphism $\nu$ of order $m$ in $M_{23}$.
\end{thm}
As formulated here, this is a slight improvement of the theorem in \cite{Sch06} due to the author of this text, removing the assumption that $P$ splits two hyperbolic planes (see Satz~6.4.2 in \cite{Moe12}).

The above result is achieved by classifying automorphic products:
\begin{thm}[\cite{Sch06}, Theorem~12.6]\label{thm:12.6}
Let $P$ be an even lattice of signature $(k,2)$ with $k\geq 4$, square-free level $m$ and $p$-ranks of the discriminant form $P'/P$ at most $k+1$. Then a completely reflective automorphic product of singular weight $-w=k/2-1$ exists on $P$ if and only if $P$ is isomorphic to one of the following lattices (the unique isomorphism class in the following lattice genera):
\begin{equation*}
\renewcommand{\arraystretch}{1.3}
\begin{array}{r|l}
\multicolumn{1}{c|}{-w} & \multicolumn{1}{c}{P}\\\hline
1 & \II_{4,2}(23^{-3})\\
2 & \II_{6,2}(11^{-4}),\,\II_{6,2}(2_{\II}^{+4}7^{-4}),\,\II_{6,2}(3^{+4}5^{-4})\\
3 & \II_{8,2}(7^{-5})\\
4 & \II_{10,2}(5^{+6}),\,\II_{10,2}(2_{\II}^{+6}3^{-6})\\
6 & \II_{14,2}(3^{-8})\\
8 & \II_{18,2}(2_{\II}^{+10})\\
12& \II_{26,2}
\end{array}
\end{equation*}
Moreover, all these lattices are of the form $P\cong\Lambda^\nu\oplus\II_{1,1}(m)\oplus\II_{1,1}$ for an element $\nu$ of square-free order $m$ in $M_{23}$.
\end{thm}
The restriction on the $p$-ranks is essential since it guarantees in particular the finiteness of the above list (see \cite{Sch06}, remark after Theorem~12.3). As before, Satz~6.4.1 in \cite{Moe12} removes the assumption that $P$ splits two hyperbolic planes.


\section{Vertex Algebras}\label{sec:va}
In this section we define the ten conformal vertex algebras $M_{\phi_\nu}$ of central charge $26$ that will serve as input of the BRST quantisation in Section~\ref{sec:brst}. For an introduction to the theory of vertex (operator) algebras and their representation theory we refer the reader to \cite{FLM88,FHL93,LL04}, for example.

Recall that a \voa{} $V=\bigoplus_{n\in\Z}V_n$ is a $\Z$-graded vertex algebra with $\dim(V_n)<\infty$ and $\dim(V_n)=0$ for $n\ll 0$. Moreover, it carries a representation of the Virasoro algebra of some central charge $c\in\C$ (see also Section~\ref{sec:brst0}) and $V_n$ is the eigenspace of $L_0$ associated with eigenvalue (or weight) $n$ for all $n\in\Z$. If we drop the assumptions of lower-boundedness of the grading and the finite-dimensionality of the graded components, we arrive at the notion of a \emph{conformal vertex algebra}. Examples of conformal vertex algebras are vertex algebras associated with even lattices. If the lattice is in addition positive-definite, then we obtain a \voa{}.

In this paper we follow the convention in \cite{DM04b} and call a \voa{} \emph{\strat{}} if it is rational (as defined in \cite{DLM97}, for example), $C_2$-cofinite (or lisse), self-contragredient (or self-dual) and of CFT-type (which imply simplicity). Rationality entails that the category of modules is semisimple with finitely many simple objects, i.e.\ irreducible modules. A \voa{} of CFT-type is $\N$-graded with $V_0=\C\vac$ where $\vac$ denotes the vacuum vector.

A \voa{} is called \emph{holomorphic} if its only irreducible module is $V$ itself.


\subsection{Heisenberg and Lattice Vertex Algebras}\label{sec:lattice}
We review Heisenberg \voa{}s and vertex algebras associated with even lattices, which are among the best-studied examples of vertex (operator) algebras.

Let $M_{\hat\h}(1,0)$ denote the Heisenberg (or free-boson) \voa{} (of level $1$) associated with the $\C$-vector space $\h$ equipped with a non-degenerate, symmetric bilinear form $\langle\cdot,\cdot\rangle$. It has central charge $\dim(\h)$ and its irreducible modules are given up to isomorphism by $M_{\hat\h}(1,\alpha)$ for $\alpha\in\h$ with conformal weights (or lowest $L_0$-weights) $\langle\alpha,\alpha\rangle/2$ (see, for example, \cite{LL04}, Section~6.3). The form $\langle\cdot,\cdot\rangle$ on $\h$ naturally induces a non-degenerate, Virasoro-invariant (see Definition~\ref{defi:virasoro} below), symmetric bilinear form on $M_{\hat\h}(1,\alpha)$ for all $\alpha\in\h$.

Over $\C$, $(\h,\langle\cdot,\cdot\rangle)$ is always isometric to $\C^d$, $d:=\dim(\h)$, equipped with the standard bilinear form, and the Heisenberg \voa{}s corresponding to $(\h,\langle\cdot,\cdot\rangle)$ and $\C^d$ are isomorphic. In the following, we shall simply write $\pi^d_\alpha:=M_{\hat\h}(1,\alpha)$ if $\dim(\h)=d$. Also note that $\pi_0^d\cong(\pi_0^1)^{\otimes d}$.

However, for the BRST construction in Section~\ref{sec:brst} we shall also need to demand that $\h$, and $M_{\hat\h}(1,\alpha)$ for all $\alpha\in\h$, come equipped with a non-degenerate, Hermitian sesquilinear form. To this end, we shall start with an $\R$-vector space $\h_\R$ equipped with a non-degenerate, symmetric bilinear form $\langle\cdot,\cdot\rangle_\R$ of signature $(r,s)$ for some $r,s\in\N$ with $r+s=d$. Then, over $\R$, $(\h_\R,\langle\cdot,\cdot\rangle_\R)$ is isometric to $\R^{(r,s)}$, the $(r+s)$-dimensional $\R$-vector space with the standard bilinear form of signature $(r,s)$. The form $\langle\cdot,\cdot\rangle_\R$ extends to a non-degenerate, symmetric bilinear form $\langle\cdot,\cdot\rangle$ and also to a non-degenerate, Hermitian sesquilinear form of signature $(r,s)$ on $\h=\h_\R\otimes_\R\C$, and these forms extend naturally to non-degenerate, Virasoro-invariant, symmetric bilinear and Hermitian sesquilinear forms on $M_{\hat\h}(1,\alpha)$, $\alpha\in\h$. In the following, we shall write $\pi^{(r,s)}_\alpha:=M_{\hat\h}(1,\alpha)$ if $(\h_\R,\langle\cdot,\cdot\rangle_\R)$ has signature $(r,s)$.

The character or graded dimension of $\pi_\alpha^d$ is given by
\begin{equation*}
\ch_{\pi_\alpha^d}(\tau)=\tr_{\pi_\alpha^d}q^{L_0-d/24}=\frac{q^{\langle\alpha,\alpha\rangle/2}}{\eta(\tau)^d}
\end{equation*}
for $\alpha\in\h$ with the Dedekind eta function $\eta$.

Let $L$ be an even lattice, i.e.\ a free abelian group $L$ of rank $d$ equipped with a non-degenerate, symmetric bilinear form $\langle\cdot,\cdot\rangle\colon L\times L\to\Z$ such that $\langle\alpha,\alpha\rangle/2\in\Z$ for all $\alpha\in L$. Let $(r,s)$ denote the signature of $\langle\cdot,\cdot\rangle$ over $\R$. We recall some well-known facts about the lattice vertex algebra $V_L$ associated with $L$ \cite{Bor86,FLM88,Don93}. $V_L$ is a conformal vertex algebra of central charge $d=r+s$. If $L$ is positive-definite, then $V_L$ is a \strat{} \voa{}, and if $L$ is unimodular, then $V_L$ is holomorphic.

The lattice vertex algebra $V_L$ contains the Heisenberg \voa{} $\pi_0^{(r,s)}$ associated with the vector space $\h=L\otimes_\Z\C$ (and $\h_\R=L\otimes_\Z\R$) as a \vosa{} and decomposes into a direct sum $V_L=\bigoplus_{\alpha\in L}\pi_\alpha^{(r,s)}$ of irreducible $\pi_0^{(r,s)}$-modules.

The irreducible $V_L$-modules up to isomorphism are $V_{\lambda+L}=\bigoplus_{\alpha\in\lambda+L}\pi_\alpha^{(r,s)}$ for $\lambda+L\in L'/L$ where $L'$ is the dual lattice. They are all simple currents, meaning that their fusion products are again irreducible, and have the fusion rules
\begin{equation*}
V_{\lambda+L}\boxtimes V_{\mu+L}\cong V_{\lambda+\mu+L}
\end{equation*}
for all $\lambda+L,\mu+L\in L'/L$, i.e.\ the fusion algebra is the group algebra $\C[L'/L]$.

In the following, let $L$ be positive-definite. Then character of $V_{\lambda+L}$ is well-defined and given by
\begin{equation*}
\ch_{V_{\lambda+L}}(\tau)=\tr_{V_{\lambda+L}}q^{L_0-d/24}=\frac{\vartheta_{\lambda+L}(\tau)}{\eta(\tau)^d}
\end{equation*}
for $\lambda+L\in L'/L$. Here, $\vartheta_{\lambda+L}(\tau)=\sum_{\alpha\in\lambda+L}q^{\langle\alpha,\alpha\rangle/2}$ denotes the usual theta series of the lattice coset $\lambda+L\in L'/L$.

We denote by $\O(L)$ the group of automorphisms (or isometries) of the lattice~$L$. The construction of the \voa{} $V_L$ involves a choice of group $2$-cocycle $\eps\colon L\times L\to\{\pm 1\}$ such that $\eps(\alpha,\beta)/\eps(\beta,\alpha)=(-1)^{\langle\alpha,\beta\rangle}$ for all $\alpha,\beta\in L$. An automorphism $\nu\in\O(L)$ together with a function $\eta\colon L\to\{\pm 1\}$ satisfying $\eta(\alpha)\eta(\beta)/\eta(\alpha+\beta)=\eps(\alpha,\beta)/\eps(\nu\alpha,\nu\beta)$ defines an automorphism $\hat{\nu}\in\O(\hat{L})\leq\Aut(V_L)$ (see, for example, \cite{FLM88,Bor92}). We call $\hat{\nu}$ a \emph{standard lift} if the restriction of $\eta$ to the fixed-point sublattice $L^\nu\subseteq L$ is trivial. All standard lifts of $\nu$ are conjugate in $\Aut(V_L)$ (see \cite{EMS20a}, Proposition~7.1). Let $\hat{\nu}$ be a standard lift of $\nu$ and suppose that $\nu$ has order $m$. If $m$ is odd or if $m$ is even and $\langle\alpha,\nu^{m/2}\alpha\rangle\in 2\Z$ for all $\alpha\in L$, then the order of $\hat{\nu}$ is also $m$. Otherwise the order of $\hat{\nu}$ is $2m$, in which case we say that $\nu$ exhibits \emph{order doubling}.


\subsection{Orbifold Theory}\label{sec:orb}
Given a suitably nice, for example \strat{}, \voa{} $V$ and a group $G\leq\Aut(V)$ of automorphisms of $V$, orbifold theory is concerned with the properties of the \fpvosa{} $V^G$ and in particular its representation theory. Recently, it was proved that if $V$ is \strat{} and $G$ is a finite, solvable group of automorphisms of $V$, then $V^G$ is also \strat{} \cite{DM97,Miy15,CM16}.

In this section we describe two special cases in which the representation theory of $V^G$ has been fully determined, i.e.\ the irreducible $V^G$-modules and the fusion rules amongst them. The first one is the cyclic orbifold theory for holomorphic \voa{}s developed in \cite{EMS20a,Moe16}. Here, $V$ is assumed to be holomorphic and $G$ to be cyclic. Secondly, we discuss the representation theory of the \voa{} $V_{L_\nu}^{\hat\nu}$ associated with the coinvariant lattice $L_\nu=(L^\nu)^\bot$ of a unimodular lattice $L$ and an automorphism $\nu\in\O(L)$. These results were obtained in \cite{Lam19} (with partial results in \cite{Moe16}, Chapter~7).

We begin with the holomorphic orbifold theory \cite{EMS20a,Moe16}. Let $V$ be a \strathol{} \voa{}, whose central charge is necessarily in $8\N$, and $G=\langle\sigma\rangle\leq\Aut(G)$ a finite, cyclic group of automorphisms of $V$ of order $m\in\Ns$.

By \cite{DLM00} there is an up to isomorphism unique irreducible $\sigma^i$-twisted $V$-module $V(\sigma^i)$ for each $i\in\Z_m$. Moreover, for all $i\in\Z_m$ the vector space $V(\sigma^i)$ admits a representation $\phi_i\colon G\to\Aut_\C(V(\sigma^i))$ of $G$ such that
\begin{equation*}
\phi_i(\sigma)Y_{V(\sigma^i)}(v,x)\phi_i(\sigma)^{-1}=Y_{V(\sigma^i)}(\sigma v,x)
\end{equation*}
for all $v\in V$. This representation is unique up to an $m$-th root of unity. We denote by $V^\sigma(i,j)$ the eigenspace of $\phi_i(\sigma)$ in $V(\sigma^i)$ corresponding to the eigenvalue $\e^{(2\pi\i)j/m}$. On $V(\sigma^0)=V$ a possible choice for $\phi_0$ is given by $\phi_0(\sigma)=\sigma$.

The \fpvosa{} $V^\sigma=V^\sigma(0,0)$ is \strat{} by \cite{DM97,Miy15,CM16} and has exactly $m^2$ irreducible modules, namely
\begin{equation*}
V^\sigma(i,j)\quad\text{for}\quad i,j\in\Z_m,
\end{equation*}
which follows from results in \cite{MT04,DRX17}. We showed that the conformal weight $\rho(V(\sigma))$ of $V(\sigma)$ is in $(1/m^2)\Z$, and we define the \emph{type} $t\in\Z_m$ of $\sigma$ by $t=m^2\rho(V(\sigma))\pmod{m}$.

For ease of presentation, let us assume in the following that $\sigma$ has type $0$, i.e.\ that $\rho(V(\sigma))\in(1/m)\Z$. (But note that the other cases were studied as well in \cite{EMS20a,Moe16}.) Then it is possible to choose the representations $\phi_i$ such that the conformal weights obey
\begin{equation*}
Q_\rho((i,j)):=\rho(V^\sigma(i,j))+\Z=ij/m+\Z=:Q_m((i,j))
\end{equation*}
and $V^\sigma$ has fusion rules
\begin{equation*}
V^\sigma(i,j)\boxtimes V^\sigma(k,l)\cong V^\sigma(i+k,j+l)
\end{equation*}
for all $i,j,k,l\in\Z_m$, i.e.\ the fusion algebra of $V^\sigma$ is the group algebra $\C[\Z_m\times\Z_m]$ (see \cite{EMS20a}, Section~5). In particular, all $V^\sigma$-modules are simple currents (see also \cite{DRX17}).

The fusion group $\Z_m\times\Z_m$ together with the quadratic form $Q_\rho=Q_m$ forms a non-degenerate \fqs{} $R(V^\sigma):=(\Z_m\times\Z_m,Q_m)$. It is isomorphic to the discriminant form of the rescaled hyperbolic lattice $\II_{1,1}(m)$, i.e.
\begin{equation*}
R(V^\sigma)\cong(\II_{1,1}(m))'/\II_{1,1}(m).
\end{equation*}

We now describe the orbifold theory for certain \voa{}s associated with coinvariant lattices \cite{Lam19}. Let $L$ be an even, positive-definite, unimodular lattice and $\nu\in\O(L)$ an isometry of $L$ of order $m$. Then $L^\nu=\{\alpha\in L\;|\;\nu\alpha=\alpha\}$ denotes the fixed-point lattice (or invariant lattice), and its orthogonal complement $L_\nu:=(L^\nu)^\bot\subseteq L$ is called \emph{coinvariant lattice}. The restriction of $\nu$ to $L_\nu$, which we shall also call $\nu$, acts fixed-point free on $L_\nu$, i.e.\ $(L_\nu)^\nu=\{0\}$. This implies that all lifts of $\nu\in\O(L_\nu)$ to $\Aut(V_{L_\nu})$ are conjugate. Let $\hat\nu$ be one such lift. It is a standard lift and has order $m$, i.e.\ no order doubling occurs.

Note however that the (up to conjugacy unique) standard lift $\phi_\nu$ of $\nu\in\O(L)$ to an automorphism in $\Aut(V_L)$ might exhibit order doubling and this will play a role in what follows.

Given the lattice \voa{} $V_{L_\nu}$ and the automorphism $\hat\nu$, we consider the \fpvosa{} $V_{L_\nu}^{\hat\nu}$. It was shown in \cite{Moe16,Lam19} that $V_{L_\nu}^{\hat\nu}$ has exactly $m^2|(L_\nu)'/L_\nu|$ irreducible modules, which are all simple currents. The exact fusion rules were determined in \cite{Lam19}. There are two cases depending on whether $\phi_\nu\in\Aut(V_L)$ exhibits order doubling or not. For simplicity let us assume that this is not the case, i.e.\ that $\langle\alpha,\nu^{m/2}\alpha\rangle\in2\Z$ for all $\alpha\in L$ if $m$ is even.

By \cite{Don93} the irreducible $V_{L_\nu}$-modules are parametrised by the lattice cosets $(L_\nu)'/L_\nu$. For $i\in\Z_m$ the irreducible $\hat\nu^i$-twisted $V_{L_\nu}$-modules were determined in \cite{DL96,BK04}. They are similarly given by $V_{\alpha+L_\nu}(\hat\nu^i)$ for $\alpha+L_\nu\in((L_\nu)'/L_\nu)^{\nu^i}$. Since $L_\nu$ is the coinvariant lattice corresponding to $\nu$, it is easy to show that $((L_\nu)'/L_\nu)^{\nu^i}=(L_\nu)'/L_\nu$, i.e.\ that $\nu$ acts trivially on $(L_\nu)'/L_\nu$. This also implies the existence of linear representations $\phi_{\alpha+L_\nu,i}\colon\langle\hat\nu\rangle\to\Aut_\C(V_{\alpha+L_\nu}(\hat\nu^i))$ satisfying the same property as the $\phi_i$ above.

With the same arguments as before, the irreducible $V_{L_\nu}^{\hat\nu}$-modules are exactly the corresponding eigenspaces
\begin{equation*}
V_{L_\nu}^{\hat\nu}(\alpha+L_\nu,i,j)\quad\text{for}\quad \alpha+L_\nu\in(L_\nu)'/L_\nu\text{ and }i,j\in\Z_m.
\end{equation*}
Again, for simplicity we only present the case when $\phi_\nu$ has type $0$, i.e.\ when $\rho(V_L(\phi_\nu))\in(1/m)\Z$. Note that $\rho(V_L(\phi_\nu))=\rho(V_{L_\nu}(\hat\nu))$. Then, after a suitable choice of the aforementioned representations, the irreducible $V_{L_\nu}^{\hat\nu}$-modules have conformal weights
\begin{equation*}
Q_\rho((\alpha+L_\nu,i,j))=\rho(V_{L_\nu}^{\hat\nu}(\alpha+L_\nu,i,j))+\Z=\frac{\langle\alpha,\alpha\rangle}{2}+\frac{ij}{m}+\Z
\end{equation*}
and fusion rules
\begin{equation*}
V_{L_\nu}^{\hat\nu}(\alpha+L_\nu,i,j)\boxtimes V_{L_\nu}^{\hat\nu}(\beta+L_\nu,k,l)\cong V_{L_\nu}^{\hat\nu}(\alpha+\beta+L_\nu,i+k,j+l)
\end{equation*}
for all $\alpha+L_\nu,\beta+L_\nu\in(L_\nu)'/L_\nu$ and $i,j,k,l\in\Z_m$, i.e.\ the fusion algebra of $V_{L_\nu}^{\hat\nu}$ is the group algebra $\C[(L_\nu)'/L_\nu\times\Z_m\times\Z_m]$. Together with the quadratic form $Q_\rho$ the fusion group forms a \fqs{}
\begin{equation*}
R(V_{L_\nu}^{\hat\nu})=(L_\nu)'/L_\nu\times R(V_L^{\phi_\nu})=(L_\nu)'/L_\nu\times(\Z_m\times\Z_m,Q_m),
\end{equation*}
which depends on the \fqs{} of $V_L^{\phi_\nu}$. Similar results hold if $\phi_\nu$ does not have type $0$ or exhibits order doubling.


\subsection{Simple-Current Extensions}\label{sec:ext}
We describe simple-current extensions of \voa{}s. A lot of progress has been made recently concerning \voa{} extensions. We shall only need the following special case, which is developed in \cite{EMS20a,Moe16}.

Let $V$ be a \strat{} \voa{} and assume that all irreducible $V$-modules are simple currents. Then the fusion algebra of $V$ is the group algebra $\C[D]$ of some finite abelian group $D$, i.e.\ the isomorphism classes of irreducible $V$-modules $\{W^\gamma\,|\,\gamma\in D\}$ can be parametrised by $D$ and
\begin{equation*}
W^\gamma\boxtimes W^\delta\cong W^{\gamma+\delta}
\end{equation*}
for all $\gamma,\delta\in D$. The identity element is given by $W^0\cong V$ and the inverse of $\gamma$ by $W^{-\gamma}\cong(W^{\gamma})'$, the contragredient module.

Now additionally assume that $V$ satisfies the \emph{positivity condition}, i.e.\ that the conformal weight $\rho(W)>0$ for any irreducible $V$-module $W\not\cong V$ and $\rho(V)=0$. Then
\begin{equation*}
Q_\rho(\gamma)=\rho(W^\gamma)+\Z\in\Q/\Z
\end{equation*}
defines a non-degenerate quadratic form on $D$, i.e.\ $(D,Q_\rho)$ is a non-degenerate finite-quadratic space.

Let $I$ be a subset of $D$. Then the direct sum
\begin{equation*}
V_I:=\bigoplus_{\gamma\in I}W^\gamma
\end{equation*}
carries an up to isomorphism unique \voa{} structure, extending the \voa{} structure of $V$ and the module structure of the $W^\alpha$, $\alpha\in I$, if and only if $I$ is an isotropic subgroup of $D$.

In this case $V_I$ is \strat{} and the irreducible $V_I$-modules are up to isomorphism given by
\begin{equation*}
X^{\alpha+I}:=\bigoplus_{\gamma\in\alpha+I}W^{\gamma}
\end{equation*}
for $\alpha+I\in I^\bot/I$. They are again all simple currents and the fusion group of $V_I$ is given by the quotient group $I^\bot/I$. In particular, $V_I$ is holomorphic if and only if $I=I^\bot$.


\subsection{Conformal Vertex Algebras of Central Charge 26}\label{sec:tenM}
Using the tools from Sections \ref{sec:orb} and \ref{sec:ext} we shall define the conformal vertex algebras $M_{\phi_\nu}$ of central charge $26$ that will serve as input for the BRST construction.

To this end consider the \strathol{} \voa{} $V_\Lambda$ of central charge $24$ associated with the Leech lattice $\Lambda$ and let $\nu\in\O(\Lambda)$ be of square-free order in $M_{23}$, i.e.\ one of the ten automorphisms from Section~\ref{sec:twisting} with orders $m=1,2,3,5,6,7,11,14,15,23$ and cycle shapes $\prod_{t\mid m}t^{b_t}=\prod_{t\mid m}t^{24/\sigma_1(m)}$. Let $\phi_\nu\in\Aut(V_\Lambda)$ be the (up to conjugacy unique) standard lift of $\nu\in\O(\Lambda)$. In the ten cases at hand, $\phi_\nu$ has order order $m$, i.e.\ no order doubling occurs, and the property that $\phi_\nu^k$ is a standard lift of $\nu^k$ for all $k\in\N$.

The conformal weight of the unique irreducible $\phi_\nu$-twisted $V_\Lambda$-module $V_\Lambda(\phi_\nu)$ is
\begin{equation*}
\rho(V_\Lambda(\phi_\nu))=\frac{1}{24}\sum_{t\mid m}b_t(t-1/t)=\frac{m-1}{m}\in\frac{1}{m}\Z.
\end{equation*}
In particular, $\phi_\nu$ has type $0$. Note that $V_\Lambda^{\phi_\nu}$ satisfies the positivity condition.

Applying the cyclic orbifold theory for holomorphic \voa{}s described in Section~\ref{sec:orb} we conclude that $V_\Lambda^{\phi_\nu}$ has exactly $m^2$ irreducible modules $V_\Lambda^{\phi_\nu}(i,j)$, $i,j\in\Z_m$, with fusion group $\Z_m\times\Z_m$ and quadratic form $Q_\rho((i,j))=\rho(V_\Lambda^{\phi_\nu}(i,j))+\Z=ij/m+\Z=Q_m((i,j))$.

Let $\II_{1,1}$ be the up to isomorphism unique even, unimodular lattice of Lorentzian signature $(1,1)$ and let $K:=\II_{1,1}(m)$ be the same lattice with the quadratic form rescaled by $m$. As mentioned above, the discriminant form $K'/K$ is as \fqs{} isomorphic to $(\Z_m\times\Z_m,Q_m)$ and in fact it is also isomorphic to $\overline{(\Z_m\times\Z_m,Q_m)}=(\Z_m\times\Z_m,-Q_m)$. (Given any \fqs{} $A$, let $\overline{A}$ be the same finite abelian group but with the quadratic form multiplied by $-1$.) We make a choice of isomorphism
\begin{equation*}
\varphi\colon K'/K\to(\Z_m\times\Z_m,-Q_m)
\end{equation*}
but shall later see that this choice is irrelevant.

Consider the conformal vertex algebra $V_K$ of central charge $2$ associated with~$K$. It has irreducible modules $V_{\alpha+K}$ for $\alpha+K\in K'/K$ and fusion group $K'/K$ \cite{Don93}.

In Table~\ref{table:ten} we collect some properties of the ten cases. Recall that $\O(\Lambda)\cong\mathrm{Co}_0$ and that the sporadic group $\mathrm{Co}_1$ is the quotient $\mathrm{Co}_0/\{\pm1\}$ of $\mathrm{Co}_0$ by its centre.
\renewcommand{\arraystretch}{1.2}
\begin{table}
\caption{The ten automorphisms $\nu\in\O(\Lambda)$ and related properties.}
\begin{tabular}{r|c|c|r|r|l|l}
$\mathrm{Co}_1$ class & Cycle shape & $\rho(V_\Lambda(\phi_\nu))$ & $k$ & $w$ & Genus $\Lambda^\nu$ & Genus $K$\\\hline
$1A$   & $1^{24}$      &$0$    &$26$&$-12$&$\II_{24,0}$                    &$\II_{1,1}$                     \\
$2A$   & $1^82^8$      &$1/2$  &$18$& $-8$&$\II_{16,0}(2_{\II}^{+8})$      &$\II_{1,1}(2_{\II}^{+2})$      \\
$3B$   & $1^63^6$      &$2/3$  &$14$& $-6$&$\II_{12,0}(3^{+6})$            &$\II_{1,1}(3^{-2})$            \\
$5B$   & $1^45^4$      &$4/5$  &$10$& $-4$&$\II_{ 8,0}(5^{+4})$            &$\II_{1,1}(5^{+2})$            \\
$6E$   & $1^22^23^26^2$&$5/6$  &$10$& $-4$&$\II_{ 8,0}(2_{\II}^{+4}3^{+4})$&$\II_{1,1}(2_{\II}^{+2}3^{-2})$\\
$7B$   & $1^37^3$      &$6/7$  & $8$& $-3$&$\II_{ 6,0}(7^{+3})$            &$\II_{1,1}(7^{-2})$            \\
$11A$  & $1^211^2$     &$10/11$& $6$& $-2$&$\II_{ 4,0}(11^{+2})$           &$\II_{1,1}(11^{-2})$           \\
$14B$  & $1.2.7.14$    &$13/14$& $6$& $-2$&$\II_{ 4,0}(2_{\II}^{+2}7^{+2})$&$\II_{1,1}(2_{\II}^{+2}7^{-2})$\\
$15D$  & $1.3.5.15$    &$14/15$& $6$& $-2$&$\II_{ 4,0}(3^{-2}5^{-2})$      &$\II_{1,1}(3^{-2}5^{+2})$      \\
$23A,B$& $1.23$        &$22/23$& $4$& $-1$&$\II_{ 2,0}(23^{+1})$           &$\II_{1,1}(23^{-2})$
\end{tabular}
\label{table:ten}
\end{table}
\renewcommand{\arraystretch}{1}

Finally, we define the conformal vertex algebra $M_{\phi_\nu}$ in the matter sector of the BRST construction as a simple-current extension of $V_\Lambda^{\phi_\nu}\otimes V_K$:
\begin{prop}\label{prop:mdef}
Let $\nu$ be of square-free order in $M_{23}$. Then the direct sum
\begin{equation*}
M_{\phi_\nu}:=\bigoplus_{\alpha+K\in K'/K}V_\Lambda^{\phi_\nu}(\varphi(\alpha+K))\otimes V_{\alpha+K}
\end{equation*}
admits the structure of a conformal vertex algebra of central charge 26.
\end{prop}
\begin{proof}
We note that $\bigoplus_{i,j\in\Z_m}V_\Lambda^{\phi_\nu}(i,j)$ is an \aia{} \cite{EMS20a,Moe16}, and so is $\bigoplus_{\alpha+K\in K'/K}V_{\alpha+K}$ \cite{DL93}, corresponding to the fact that all the irreducible modules are simple currents.

An \aia{} \cite{DL93} is a generalisation of a conformal vertex algebra associated with some \fqs{}. Conformal vertex algebras are recovered if the $L_0$-grading is integral and the quadratic form trivial. The axioms of an \aia{} also include a grading-compatibility condition that relates the bilinear form associated with the quadratic form to the $L_0$-grading. Under mild assumptions (see, for example, Remark~3.1.5 in \cite{Moe16}) this guarantees that if an \aia{} has integral $L_0$-grading, this bilinear form vanishes. This does not quite mean, however, that the quadratic form vanishes.\footnote{Indeed, every quadratic form $Q$ (on some finite, abelian group $D$) has a unique associated bilinear form $B_Q$. On the other hand, given a finite bilinear form $B$, there are $|D/2D|$ many quadratic forms $Q$ with $B_Q=B$.} Some \aia{}s satisfy an additional evenness condition. In that case, the quadratic form itself is related to the $L_0$-grading, and hence an integral $L_0$-grading does imply that the quadratic form vanishes.

Now, the tensor-product \aia{} of central charge $26$
\begin{equation*}
\Big(\bigoplus_{i,j\in\Z_m}V_\Lambda^{\phi_\nu}(i,j)\Big)\otimes\Big(\bigoplus_{\alpha+K\in K'/K}V_{\alpha+K}\Big)=\bigoplus_{\substack{i,j\in\Z_m\\\alpha+K\in K'/K}}V_\Lambda^{\phi_\nu}(i,j)\otimes V_{\alpha+K}
\end{equation*}
is an \aia{} with associated \fqs{}
\begin{equation*}
(\Z_m\times\Z_m,-Q_m)\times K'/K.
\end{equation*}
It was shown in \cite{EMS20a,Moe16} that the first \aia{} satisfies the evenness condition, and for the second this follows by definition of lattice \aia{}s \cite{DL93}. Hence, also the tensor product satisfies evenness.

Clearly, by definition of $\varphi$, the abelian intertwining subalgebra $M_{\phi_\nu}$ defined by the subgroup of all elements of the form
\begin{equation*}
(\varphi(\gamma),\gamma)\quad\text{for}\quad\gamma\in K'/K
\end{equation*}
has integral $L_0$-grading. Hence, the quadratic form for $M_{\phi_\nu}$ vanishes and $M_{\phi_\nu}$ is a conformal vertex algebra.
\end{proof}
We shall see in Proposition~\ref{prop:choiceindep} that $M_{\phi_\nu}$ is up to isomorphism independent of the choice of $\varphi$.

For the remainder of this section we study the properties of
\begin{equation*}
M_{\phi_\nu}=\bigoplus_{\alpha\in K'}V_\Lambda^{\phi_\nu}(\varphi(\alpha+K))\otimes\pi^{(1,1)}_\alpha,
\end{equation*}
which is clearly graded by $K'$. In the following we shall see that $M_{\phi_\nu}$ is actually graded by $L'$ where
\begin{equation*}
L:=\Lambda^\nu\oplus K=\Lambda^\nu\oplus \II_{1,1}(m)
\end{equation*}
is a lattice of rank $k$ and signature $(k-1,1)$. Indeed, recall that $V_{\Lambda_\nu}^{\hat\nu}$ is the orbifold \voa{} associated with the coinvariant lattice $\Lambda_\nu$ and the up to conjugacy unique lift $\hat\nu\in\Aut(V_{\Lambda_\nu})$ of $\nu\in\O(\Lambda_\nu)$. It is not difficult to see that $V_{\Lambda_\nu}^{\hat\nu}$ and the lattice \voa{} $V_{\Lambda^\nu}$ form a dual pair in $V_\Lambda^{\phi_\nu}$, i.e.\ they are mutual commutants (or centralisers), intersect trivially, i.e.\ $V_{\Lambda_\nu}^{\hat\nu}\cap V_{\Lambda^\nu}=\C\vac$, and generate a full vertex operator subalgebra of $V_\Lambda^{\phi_\nu}$ isomorphic to $V_{\Lambda_\nu}^{\hat\nu}\otimes V_{\Lambda^\nu}$.

This implies that we can decompose $V_\Lambda^{\phi_\nu}$ and any of its modules into a direct sum of irreducible $V_{\Lambda_\nu}^{\hat\nu}\otimes V_{\Lambda^\nu}$-modules. First we observe that because fixed-point sublattices are always primitive sublattices and because $\Lambda$ is unimodular, there is a natural isomorphism of \fqs{}s
\begin{equation*}
\psi\colon(\Lambda^\nu)'/\Lambda^\nu\to\overline{(\Lambda_\nu)'/\Lambda_\nu}
\end{equation*}
such that
\begin{equation*}
\Lambda\cong\bigcup_{\alpha+\Lambda^\nu\in(\Lambda^\nu)'/\Lambda^\nu}\psi(\alpha+\Lambda^\nu)\oplus(\alpha+\Lambda^\nu)
\end{equation*}
(see, for example, Proposition~1.2 in \cite{Ebe13}). Hence,
\begin{equation*}
V_\Lambda\cong\bigoplus_{\alpha+\Lambda^\nu\in(\Lambda^\nu)'/\Lambda^\nu}V_{\psi(\alpha+\Lambda^\nu)}\otimes V_{\alpha+\Lambda^\nu}.
\end{equation*}
This can be used to show that
\begin{equation*}
V_\Lambda^{\phi_\nu}(i,j)\cong\bigoplus_{\alpha+\Lambda^\nu\in(\Lambda^\nu)'/\Lambda^\nu}V_{\Lambda_\nu}^{\hat\nu}(\psi(\alpha+\Lambda^\nu),i,j)\otimes V_{\alpha+\Lambda^\nu}
\end{equation*}
for all $i,j\in\Z_m$ (see \cite{Lam19}, proof of Theorem~5.3).

Inserting the above into the definition of $M_{\phi_\nu}$ and defining the isomorphism $\chi:=(\psi,\varphi)$ of \fqs{}s
\begin{equation*}
\chi\colon L'/L\longrightarrow\overline{(\Lambda_\nu)'/\Lambda_\nu}\times(\Z_m\times\Z_m,-Q_m)
\end{equation*}
where $L=\Lambda^\nu\oplus K=\Lambda^\nu\oplus \II_{1,1}(m)$ we can decompose $M_{\phi_\nu}$ as simple-current extension of $V_{\Lambda_\nu}^{\hat\nu}\otimes V_L$:
\begin{prop}\label{prop:Mdecomp}
Let $\nu$ be of square-free order in $M_{23}$. Then the conformal vertex algebra $M_{\phi_\nu}$ decomposes as
\begin{equation*}
M_{\phi_\nu}\cong\bigoplus_{\gamma+L\in L'/L}V_{\Lambda_\nu}^{\hat\nu}(\chi(\gamma+L))\otimes V_{\gamma+L}.
\end{equation*}
\end{prop}
\begin{proof}
With the above observation we can decompose $M_{\phi_\nu}$ as $V_{\Lambda_\nu}^{\hat\nu}\otimes V_L$-module
\begin{align*}
M_{\phi_\nu}&=\bigoplus_{\beta+K\in K'/K}V_\Lambda^{\phi_\nu}(\varphi(\beta+K))\otimes V_{\beta+K}\\
&\cong\bigoplus_{\beta+K\in K'/K}\bigoplus_{\alpha+\Lambda^\nu\in(\Lambda^\nu)'/\Lambda^\nu}V_{\Lambda_\nu}^{\hat\nu}(\psi(\alpha+\Lambda^\nu),\varphi(\beta+K))\otimes V_{\alpha+\Lambda^\nu}\otimes V_{\beta+K}\\
&\cong\bigoplus_{\gamma+L\in L'/L}V_{\Lambda_\nu}^{\hat\nu}(\chi(\gamma+L))\otimes V_{\gamma+L},
\end{align*}
which proves the assertion.
\end{proof}
The proposition implies in particular that $M_{\phi_\nu}$ is graded by $L'$, i.e.\
\begin{equation*}
M_{\phi_\nu}=\bigoplus_{\alpha\in L'}M_{\phi_\nu}(\alpha)=\bigoplus_{\alpha\in L'}V_{\Lambda_\nu}^{\hat\nu}(\chi(\alpha+L))\otimes\pi^{(k-1,1)}_\alpha
\end{equation*}
with $M_{\phi_\nu}(\alpha)=V_{\Lambda_\nu}^{\hat\nu}(\chi(\alpha+L))\otimes\pi^{(k-1,1)}_\alpha$ for all $\alpha\in L'$.

Note that $(V_{\Lambda_\nu}^{\hat\nu})_1=\{0\}$ since $\Lambda_\nu\subseteq\Lambda$ has no vectors $\alpha$ of norm $\langle\alpha,\alpha\rangle/2=1$ and $\nu$ acts fixed-point free on $\Lambda_\nu\otimes\C$. This plays a role in Section~\ref{sec:tenbkmas} when we determine a Cartan subalgebra for the Lie algebra obtained as quantisation of $M_{\phi_\nu}$.

In the following we shall prove that the conformal vertex algebra $M_{\phi_\nu}$ is up to isomorphism independent of the isomorphism $\chi$ and hence in particular of the choice of $\varphi$.
\begin{lem}\label{lem:surjective}
Let $\nu$ be of square-free order $m$ in $M_{23}$ and $L=\Lambda^\nu\oplus\II_{1,1}(m)$. Then the natural group homomorphism $\O(L)\to\O(L'/L)$ is surjective.
\end{lem}
\begin{proof}
Of the ten lattices $L=\Lambda^\nu\oplus\II_{1,1}(m)$ all but one fulfil the assumptions of Theorem~1.14.2 in \cite{Nik80}, which implies the assertion. The lattice for $m=23$ of genus $\II_{3,1}(23^{-3})$ is covered by Corollary~7.8 in \cite{MM09}, Chapter~VIII.
\end{proof}
\begin{prop}\label{prop:choiceindep}
Let $\nu$ be of square-free order $m$ in $M_{23}$. Then the isomorphism class of $M_{\phi_\nu}$ does not depend on the isomorphism $\chi\colon L'/L\to\overline{(\Lambda_\nu)'/\Lambda_\nu}\times(\Z_m\times\Z_m,-Q_m)$ and is in particular independent of the choice of the isomorphism $\varphi\colon K'/K\to(\Z_m\times\Z_m,-Q_m)$.
\end{prop}
\begin{proof}
As in the proof of Lemma~3.1 in \cite{HS14}, the decomposition in Proposition~\ref{prop:Mdecomp} and Lemma~\ref{lem:surjective} imply the assertion.
\end{proof}


\subsection{Characters}\label{sec:chars}
We describe the characters of the irreducible modules of the orbifold \voa{}s $V_{L_\nu}^{\hat\nu}$ from Section~\ref{sec:orb} and, more specifically, of $V_{\Lambda_\nu}^{\hat\nu}$ where $\nu$ is one of the ten automorphisms of square-free order in $M_{23}$. We then show that the latter form a vector-valued modular form obtained as lift of a certain eta product associated with $\nu$.

The \voa{} $V_{L_\nu}^{\hat\nu}$ is \strat{} of central charge $\rk(L_\nu)$ and has group-like fusion with fusion group $R(V_{L_\nu}^{\hat\nu})=(L_\nu)'/L_\nu\times(\Z_m\times\Z_m,Q_m)$. The corresponding characters
\begin{equation*}
\ch_{V_{L_\nu}^{\hat\nu}(\alpha+L_\nu,i,j)}(\tau)=\tr_{V_{L_\nu}^{\hat\nu}(\alpha+L_\nu,i,j)}q^{L_0-c/24},
\end{equation*}
$q=\e^{(2\pi\i)\tau}$, for $\alpha+L_\nu\in(L_\nu)'/L_\nu$ and $i,j\in\Z_m$ satisfy Zhu's modular invariance \cite{Zhu96}, i.e.\ they form a vector-valued modular form of weight $0$ for Zhu's representation
\begin{equation*}
\rho_{V_{L_\nu}^{\hat\nu}}\colon\SLZ\to\GL(\C[R(V_{L_\nu}^{\hat\nu})])
\end{equation*}
that is holomorphic on the upper half-plane $\H$ but may have poles at the cusp $\infty$. Since all irreducible $V_{L_\nu}^{\hat\nu}$-modules are simple currents, Zhu's representation takes a very simple form (see \cite{EMS20a}, Theorem~3.4, \cite{Moe16}, Proposition~2.2.6):
\begin{align*}
\rho_{V_{L_\nu}^{\hat{\rho}}}(S)_{(\alpha+L_\nu,i,j),(\beta+L_\nu,k,l)}&=\frac{1}{m\sqrt{|(L_\nu)'/L_\nu|}}\e^{-(2\pi\i)(\langle\alpha,\beta\rangle+(il+jk)/m)},\\
\rho_{V_{L_\nu}^{\hat{\rho}}}(T)_{(\alpha+L_\nu,i,j),(\beta+L_\nu,k,l)}&=\delta_{(\alpha+L_\nu,i,j),(\beta+L_\nu,k,l)}\e^{(2\pi\i)(\langle\alpha,\alpha\rangle/2+ij/m-c/24)}
\end{align*}
for the standard generators $S,T\in\SLZ$.

The characters of the irreducible $V_{L_\nu}^{\hat\nu}$-modules $V_{L_\nu}^{\hat\nu}(\alpha+L_\nu,0,j)$, i.e.\ those stemming from the irreducible untwisted $V_{L_\nu}$-modules, can be computed directly. In fact, we shall be able to express them explicitly in terms of theta series and the eta function. Since their modular properties are explicitly known, we can then determine the full vector-valued character of $V_{L_\nu}^{\hat\nu}$ by applying modular transformations.

More precisely, in order to compute the characters of the irreducible modules $V_{L_\nu}^{\hat\nu}(\alpha+L_\nu,0,j)$ we first consider the twisted trace functions \cite{DLM00}
\begin{equation*}
T_{\alpha+L_\nu,i,j}(\tau):=\tr_{V_{\alpha+L_\nu}(\hat\nu^i)}\phi_{\alpha+L_\nu,i}(\hat\nu^j)q^{L_0-c/24}
\end{equation*}
for all $\alpha+L_\nu\in(L_\nu)'/L_\nu$ and $i,j\in\Z_m$ where $\phi_{\alpha+L_\nu,i}$ is the choice of representation of $\langle\hat\nu\rangle$ on $V_{\alpha+L_\nu}(\hat\nu^i)$ described in Section~\ref{sec:orb}. It follows directly from the definition of the irreducible $V_{L_\nu}^{\hat\nu}$-modules that
\begin{equation*}
\ch_{V_{L_\nu}^{\hat\nu}(\alpha+L_\nu,i,j)}(\tau)=\frac{1}{m}\sum_{k\in\Z_m}\e^{-(2\pi\i)jk/m}T_{\alpha+L_\nu,i,k}(\tau)
\end{equation*}
for all $\alpha+L_\nu\in(L_\nu)'/L_\nu$ and $i,j\in\Z_m$.

Since the action of $\langle\hat\nu\rangle$ on the untwisted $V_{L_\nu}$-modules $V_{\alpha+L_\nu}$ for all $\alpha+L_\nu\in(L_\nu)'/L_\nu$ can be explicitly determined, it is possible to compute $T_{\alpha+L_\nu,0,j}(\tau)$ and hence $\ch_{V_{L_\nu}^{\hat\nu}(\alpha+L_\nu,0,j)}(\tau)$ for all $\alpha+L_\nu\in(L_\nu)'/L_\nu$ and $j\in\Z_m$.

Now consider the \voa{} $V_{\Lambda_\nu}^{\hat\nu}$ where $\Lambda$ is the Leech lattice and $\nu$ is one of the ten automorphisms of square-free order in $M_{23}$. Recall that for a lattice automorphism of cycle shape $\prod_{t\mid m}t^{b_t}$, $b_t\in\Z$, the associated eta product is $\eta_\nu(\tau)=\prod_{t\mid m}\eta(t\tau)^{b_t}$. Also, for any subset $S$ of a positive-definite lattice the corresponding theta series is defined as $\vartheta_S(\tau):=\sum_{\alpha\in S}q^{\langle\alpha,\alpha\rangle/2}$.
\begin{prop}
Let $\nu$ be of square-free order $m$ in $M_{23}$. Assume that the representations $\phi_{\alpha+L_\nu,0}$ of $\langle\hat\nu\rangle$ on the irreducible $V_{\Lambda_\nu}$-modules are chosen as in Section~\ref{sec:orb}. Then
\begin{equation*}
T_{\alpha+\Lambda_\nu,0,j}(\tau)=\frac{\vartheta_{(\alpha+\Lambda_\nu)^{\nu^j}}(\tau)}{\eta_{\nu^j}(\tau)}
\end{equation*}
for all $\alpha+\Lambda_\nu\in(\Lambda_\nu)'/\Lambda_\nu$ and $j\in\Z_m$ where $(\alpha+\Lambda_\nu)^{\nu^j}$ are the vectors in the lattice coset $\alpha+\Lambda_\nu$ invariant under $\nu^j$.
\end{prop}
\begin{proof}
The somewhat technical proof can be found in \cite{Moe16}, Proposition~7.5.9 and Lemma~7.6.8. For the assertion to hold, the actions of $\langle\hat\nu\rangle$ on the irreducible $V_{\Lambda_\nu}$-modules have to be sufficiently nice. In general, the theta series in the above expression would be modified by some function $(\Lambda_\nu)'\to\{\pm1\}$.
\end{proof}

The above proposition and the preceding discussion allow us to compute the vector-valued character of $V_{\Lambda_\nu}^{\hat\nu}$. By multiplying by a suitable power of the eta function we make the character transform under the more standard Weil representation rather than Zhu's representation:
\begin{prop}\label{prop:F1}
Let $\nu$ be of square-free order $m$ in $M_{23}$. Then
\begin{equation*}
\ch_{V_{\Lambda_\nu}^{\hat\nu}(\alpha+\Lambda_\nu,i,j)}(\tau)/\eta(\tau)^{\rk(\Lambda^\nu)}
\end{equation*}
for $\alpha+\Lambda_\nu\in(\Lambda_\nu)'/\Lambda_\nu$ and $i,j\in\Z_m$ are the components of a vector-valued modular form, holomorphic on $\H$ but with possible poles at the cusp $\infty$, of weight $w=-\rk(\Lambda^\nu)/2=1-k/2\in\Z_{<0}$ for the Weil representation of $\SLZ$ on $\C[R(V_{\Lambda_\nu}^{\hat\nu})]$.
\end{prop}
\begin{proof}
By Corollary~2.2.13 in \cite{Moe16}, the $\ch_{V_{\Lambda_\nu}^{\hat\nu}(\alpha+\Lambda_\nu,i,j)}(\tau)\eta(\tau)^{\rk(\Lambda_\nu)}$ for $\alpha+\Lambda_\nu\in(\Lambda_\nu)'/\Lambda_\nu$ and $i,j\in\Z_m$ form a vector-valued modular form of weight $\rk(\Lambda_\nu)/2$ for the Weil representation of $\SLZ$ on $\C[R(V_{\Lambda_\nu}^{\hat\nu})]$. Dividing by $\Delta(\tau)=\eta(\tau)^{24}$, which is modular of weight $12$, yields the assertion.
\end{proof}

In the following we shall see that the vector-valued modular form from Proposition~\ref{prop:F1} is exactly the vector-valued modular form $F$ obtained in \cite{Sch04b,Sch06,Sch08} as lift of a certain scalar-valued modular form associated with $\nu$ (see also Section~\ref{sec:twisting}).

We consider the eta product
\begin{equation*}
f(\tau):=\frac{1}{\eta_\nu(\tau)}=\prod_{t\mid m}\eta(t\tau)^{-24/\sigma_1(m)}
\end{equation*}
associated with the cycle shape of $\nu\in\O(\Lambda)$. Products of rescaled eta functions are sometimes modular forms.

To describe this in more detail, we define the Dirichlet character $\chi_s$ for $s\in\Ns$ as the Kronecker symbol $\chi_s(j):=(j/s)$, $j\in\Z$. Note that if $s$ is an odd prime, then $\chi_s$ is a character modulo~$s$. For $s=1$ we get the trivial character. Given a quadratic Dirichlet character $\chi\colon\Z\to\{\pm1\}$ of some modulus $k\in\Ns$ we can view it as a character $\chi\colon\Gamma_0(k)\to\{\pm1\}$ on the congruence subgroup $\Gamma_0(k)$ by setting $\chi(M):=\chi(a)=\chi(d)$ for $M=\left(\begin{smallmatrix}a&b\\c&d\end{smallmatrix}\right)\in\Gamma_0(k)$. Then clearly, $\chi$ is also a character on $\Gamma_0(l)$ for any multiple $l$ of $k$.

Theorem~6.2 in \cite{Bor00} implies:
\begin{lem}\label{lem:fmodform}
Let $\nu$ be of square-free order $m$ in $M_{23}$. Then $f(\tau)$ is a modular form, holomorphic on $\H$ but with possible poles at the cusps, of weight $w=1-k/2=-12\sigma_0(m)/\sigma_1(m)\in\Z_{<0}$ for the congruence subgroup $\Gamma_0(m)$ and character $\chi_s$ where $s=s(m)\in\Ns$ is chosen such that $s\prod_{t\mid m}t^{24/\sigma_1(m)}$ is a rational square, i.e.\
\begin{equation*}
s(m)=
\begin{cases}
m&\text{if }m=7,23,\\
1&\text{otherwise.}
\end{cases}
\end{equation*}
\end{lem}
Note that as described above, $\chi_s$ is indeed a character on $\Gamma_0(m)$ and it is the trivial character except for $m=7,23$.

Consider now the lattice $L=\Lambda^\nu\oplus K$ and its discriminant form $L'/L$. It has level $m$ and even signature. For any \fqs{} $D$ of even signature and level $N$ we define
\begin{equation*}
\chi_D(j):=\left(\frac{j}{\left|D\right|}\right)e\left(\left(j-1\right)\oddity(D)/8\right),
\end{equation*}
$j\in\Z$, which is a quadratic Dirichlet character modulo~$N$ (see, for example, Section~6 in \cite{Sch06}). If 4 does not divide the level $N$, for instance if $N$ is square-free, then the character simplifies and becomes
\begin{equation*}
\chi_D(j)=\left(\frac{j}{\left|D\right|}\right).
\end{equation*}
Using elementary properties of the Kronecker symbol we find:
\begin{lem}
Let $\nu$ be of square-free order $m$ in $M_{23}$ and $L=\Lambda^\nu\oplus\II_{1,1}(m)$. Then $\chi_{L'/L}=\chi_s$ for $s=s(m)$ as defined in Lemma~\ref{lem:fmodform}.
\end{lem}

This lemma allows us to lift $f(\tau)=1/\eta_\nu(\tau)$ to a vector-valued modular form for the (dual) Weil representation on $\C[L'/L]$.
\begin{prop}\label{prop:F2}
Let $\nu$ be of square-free order $m$ in $M_{23}$. Then
\begin{equation*}
F_{\alpha+L}(\tau):=\!\!\!\!\sum_{M\in\Gamma_0(m)\backslash\SLZ}\!\!\!\!(c\tau+d)^{-w}\frac{1}{\eta_\nu(M.\tau)}\overline\rho_{L'/L}(M^{-1})_{\alpha+L,0+L}
\end{equation*}
for $\alpha+L\in L'/L$ defines a vector-valued modular form $F$, holomorphic on $\H$ but with possible poles at the cusp $\infty$, of weight $w=1-k/2=-12\sigma_0(m)/\sigma_1(m)$ for the dual Weil representation $\overline\rho_{L'/L}$ of $\SLZ$ on $\C[L'/L]$. Moreover, $F$ is invariant under the automorphisms of the \fqs{} $L'/L$.
\end{prop}
\begin{proof}
Given a \fqs{} $D$ of even signature and level dividing $N$ and a modular form $f$ of weight $w\in\Z$ for $\Gamma_0(N)$ and character $\chi_D$ it was shown in \cite{Sch06}, Theorem~6.2, that
\begin{equation*}
F_\gamma(\tau):=\!\!\!\!\sum_{M\in\Gamma_0(N)\backslash\SLZ}\!\!\!\!(c\tau+d)^{-w}f(M.\tau)\overline\rho_D(M^{-1})_{\gamma,0},
\end{equation*}
$\gamma\in D$, are the components of a vector-valued modular form $F$ of weight $w$ for the dual Weil representation $\overline\rho_D$ of $\SLZ$ on $\C[D]$, which is invariant under the automorphisms of the \fqs{} $D$. $F$ is called the lift of $f$ with trivial support.

Applying this result to $D=L'/L$ of level $m$ and $f(\tau)=1/\eta_\nu(\tau)$, which is a modular form of weight $w$ for $\Gamma_0(m)$ with character $\chi_{L'/L}$, yields the assertion.
\end{proof}

The main result of this section is the following proposition, which shows that the two vector-valued modular forms from Propositions \ref{prop:F1} and \ref{prop:F2} are equal. Recall that there is an isomorphism $\chi\colon L'/L\to\overline{(\Lambda_\nu)'/\Lambda_\nu}\times(\Z_m\times\Z_m,-Q_m)$.
\begin{prop}\label{prop:F1eqF2}
Let $\nu$ be of square-free order in $M_{23}$. Then
\begin{equation*}
\ch_{V_{\Lambda_\nu}^{\hat{\nu}}(\chi(\alpha+L))}(\tau)/\eta(\tau)^{\rk(\Lambda^\nu)}=F_{\alpha+L}(\tau)
\end{equation*}
for all $\alpha+L\in L'/L$.
\end{prop}
\begin{proof}
We consider the vector-valued modular form $G$ with components $G_{\alpha+L}(\tau):=\ch_{V_{\Lambda_\nu}^{\hat\nu}(\chi(\alpha+L))}(\tau)/\eta(\tau)^{\rk(\Lambda^\nu)}$ for $\alpha+L\in L'/L$. We have to prove that $F=G$.

By Proposition~\ref{prop:F2}, $F$ is a vector-valued modular form of weight $w$ for $\overline{\rho}_{L'/L}$. Proposition~\ref{prop:F1} states that the functions $G_{\alpha+L}(\tau)=\ch_{V_{\Lambda_\nu}^{\hat{\nu}}(\chi(\alpha+L))}/\eta(\tau)^{\rk(\Lambda^\nu)}$ form a vector-valued modular form of weight $w$ for the Weil representation on the fusion group $(\Lambda_\nu)'/\Lambda_\nu\times(\Z_m\times\Z_m,Q_m)\cong\overline{L'/L}$ (via $\chi$), which is the same as the dual Weil representation $\overline{\rho}_{L'/L}$ on $L'/L$.

Hence, $F$ and $G$ are both vector-valued modular forms of the same negative weight $w$ for $\overline{\rho}_{L'/L}$, and they are both holomorphic on $\H$ with possible poles at the cusp $\infty$.

We compute the $q$-expansions of $F$ and $G$ explicitly and verify that the singular coefficients are identical. The lift $F$ takes a very simple form (see Proposition~\ref{prop:simplelift} below) and hence its $q$-expansion can be easily determined using the well-known $q$-expansion of the eta function. The computation of the characters of the irreducible $V_{\Lambda_\nu}^{\hat{\nu}}$-modules, which enter $G$, is described at the beginning of this section. These calculations are performed in \texttt{Sage} and \texttt{Magma} \cite{Sage,Magma}.

Then $F-G$ is a modular form of negative weight, which has no singular terms, i.e.\ which is finite at the cusp $\infty$, and therefore has to vanish by the valence formula (see, for example, \cite{HBJ94}, Theorem~I.4.1). Hence, $F=G$.
\end{proof}

We comment on some special properties of the modular form $F$:
\begin{rem}
\begin{enumerate}[wide]
\item\label{item:rem1} Since $L=\Lambda^\nu\oplus\II_{1,1}(m)$ and $P=L\oplus\II_{1,1}$ have the same discriminant form $L'/L\cong P'/P$, we can view $F$ also as a vector-valued modular form for the dual Weil representation $\overline{\rho}_{P'/P}$ on $\C[P'/P]$. As such $F$ is \emph{completely reflective} (as defined in \cite{Sch06}, Section~9). Note that the lattice $P$ has signature $(k,2)$ and $F$ weight $w=1-k/2$ with $k\geq 4$ even.

Exactly such vector-valued modular forms are classified in \cite{Sch06}. Theorems \ref{thm:12.6} and \ref{thm:12.7} state the corresponding results for automorphic products and \BKMa{}s, respectively.

In the ten cases at hand complete reflectivity means that singular terms in the $q$-expansion of $F$ appear exactly in the components $F_{\alpha+P}(\tau)$, $\alpha+P\in P'/P$, with $\langle\alpha,\alpha\rangle/2=1/d\pmod{1}$ and $d\cdot(\alpha+L)=0+L$ for $d\mid m$ and in such a component the only singular term is $1\cdot q^{-1/d}$.

\item\label{item:rem2} As completely reflective modular form, $F$ is in particular \emph{symmetric}, i.e.\ invariant under the automorphisms of the \fqs{} $P'/P\cong L'/L$ (see Section~9 in \cite{Sch06} and note that $m$, the level of $P$ or $L$, is square-free). This also follows immediately from Proposition~\ref{prop:F2}.

Then, the characters of the irreducible $V_{\Lambda_\nu}^{\hat{\nu}}$-modules are invariant under the automorphisms of the fusion group $R(V_{\Lambda_\nu}^{\hat{\nu}})$ as \fqs{}. In particular, the characters $\ch_{V_{\Lambda_\nu}^{\hat{\nu}}(\chi(\alpha+L))}(\tau)$ do not depend on the choice of the isomorphism $\chi\colon L'/L\to (\Lambda_\nu)'/\Lambda_\nu\times(\Z_m\times\Z_m,-Q_m)$ (cf.\ Proposition~\ref{prop:choiceindep}).

\item\label{item:rem3} The automorphic product $\Psi_{\phi_\nu}$ on $P$, which is the denominator identity of $\g_{\phi_\nu}$, is constructed in \cite{Sch04} precisely as the Borcherds lift of the modular form~$F$.
\end{enumerate}
\label{rem:vvmfF}
\end{rem}

In the following we present a nice explicit formula for the components of the vector-valued modular form $F$ based on Theorem~6.5 in \cite{Sch06}. This was already stated in \cite{Sch04b}, Proposition~9.5. We give a proof for completeness. For $d\in\Ns$, we decompose $f(\tau/d)$, which has an expansion in $q^{1/d}$, as $f(\tau/d)=g_{d,0}(\tau)+\ldots+g_{d,d-1}(\tau)$ where $g_{d,j}(\tau)$ transforms under $T$ like $g_{d,j}(\tau+1)=\e^{(2\pi\i)j/d}g_{d,j}(\tau)$.
\begin{prop}\label{prop:simplelift}
Let $\nu$ be of square-free order $m$ in $M_{23}$. Then
\begin{equation*}
F_{\alpha+L}(\tau)=\sum_{d\mid m}\delta_{\alpha\in L'\cap \frac{1}{d}L}\,g_{d,j_{\alpha+L,d}}(\tau)
\end{equation*}
for all $\alpha+L\in L'/L$ with $j_{\alpha+L,d}\in\Z_d$ such that $-j_{\alpha+L,d}/d=\langle\alpha,\alpha\rangle/2\pmod{1}$.
\end{prop}
\begin{proof}
Explicit formulae for the components of lifts of scalar-valued modular forms are given in \cite{Sch06}, Theorem~6.5: let $F$ be the lift of a scalar-valued modular form $f$ for the dual Weil representation $\overline{\rho}_D$ on some discriminant form $D$ of even signature and level dividing $N$. Assume that $N$ is square-free. Then for $\gamma\in D$,
\begin{equation*}
F_\gamma(\tau)=\sum_{c\mid N}\delta_{\gamma\in D_c}\xi_{\frac{N}{c}}\frac{1}{\sqrt{|D_c|}}c\,h_{c,j_{\gamma,c}}(\tau)
\end{equation*}
where for $c\mid N$ the $\xi_c$ are certain factors of unit modulus and the $h_{c,j}$, $j\in\Z_c$, are obtained from $f_{N/c}(\tau)$ in the same manner as the $g_{c,j}$ are obtained from $f(\tau/c)$. The $f_c(\tau)$ for $c\mid N$ are defined as $f_c(\tau):=(c\tau+d)^{-w}f(M_c.\tau)$ where the matrices $M_c=\left(\begin{smallmatrix}a&b\\c&d\end{smallmatrix}\right)\in\SLZ$ are chosen such that $d=1\pmod{c}$ and $d=0\pmod{N/c}$. Finally, $D_c=\{\gamma\in D\,|\, c\gamma=0\}$.

Returning to the specific cases at hand, with $D=L'/L$ of square-free level~$m$ and $f=1/\eta_\nu$, using that the modular-transformation properties of the eta function and rescaled eta functions are explicitly known (see, for example, \cite{Sch09}, Proposition~6.2), we compute the $f_c(\tau)$. Due to the highly symmetric nature of the eta product $f(\tau)=1/\eta_\nu(\tau)=\prod_{t\mid m}\eta(t\tau)^{-24/\sigma_1(m)}$ one obtains that
\begin{equation*}
f_{m/c}(\tau)=\psi_{m/c}f(\tau/c)\prod_{t\mid m}(t,c)^{12/\sigma_1(m)}
\end{equation*}
for some phase factor $\psi_{m/c}$ of unit modulus and hence
\begin{equation*}
h_{c,j}(\tau)=g_{c,j}(\tau)\psi_{\frac{m}{c}}\prod_{t\mid m}(t,c)^{12/\sigma_1(m)}.
\end{equation*}
The cardinality of $D_c=(L'\cap(1/c)L)/L$ is $|D_c|=c^2\prod_{t\mid m}(t,c)^{24/\sigma_1(m)}$. Consequently all factors of non-unit modulus cancel and
\begin{equation*}
F_{\alpha+L}(\tau)=\sum_{c\mid m}\delta_{\alpha+L\in D_c}\xi_{\frac{m}{c}}\psi_{\frac{m}{c}}g_{c,j_{\alpha+L,c}}(\tau).
\end{equation*}
A case-by-case study reveals that $\xi_c\psi_c=1$ for all $m=1,2,3,5,6,7,11,14,15,23$ and all $c\mid m$, completing the proof.
\end{proof}
The results we just proved about the vector-valued modular form $F$ will play an important role in Section~\ref{sec:tenbkmas} when we relate its Fourier coefficients to the dimensions of the graded components of the Lie algebra obtained as BRST quantisation of $M_{\phi_\nu}$.


\section{BRST Construction}\label{sec:brst}
In this section we describe the BRST quantisation of certain Virasoro representations $M$ of central charge $26$ and study the resulting physical space if $M$ is additionally a conformal vertex algebra, admits an invariant bilinear form or carries a certain representation of the Heisenberg \voa{} \cite{Fei84,FGZ86,Zuc89,LZ91,LZ93} (based on the semi-infinite cohomology of graded Lie algebras \cite{Fei84,FGZ86}). To some extent, we follow the presentation in \cite{Car12b}, Section~3.

Then we apply the BRST quantisation to the conformal vertex algebras $M_{\phi_\nu}$ from Section~\ref{sec:tenM} and show that the resulting \BKMa{}s are isomorphic to the ten twisted Fake Monster Lie algebras $g_{\phi_\nu}$ in Section~\ref{sec:bkma}.


\subsection{BRST Quantisation}\label{sec:brst0}
We describe the BRST quantisation of Virasoro representations of central charge $26$.

A representation of the Virasoro algebra is a complex vector space $V$ equipped with operators $L_n$, $n\in\Z$, and $K$ in $\End(V)$ satisfying the Virasoro relations
\begin{equation*}
[L_m,L_n]=(m-n)L_{m+n}+\frac{m^3-m}{12}\delta_{m+n,0}K\quad\text{and}\quad[L_n,K]=0
\end{equation*}
for all $m,n\in\Z$.

We define some important notions:
\begin{defi}\label{defi:virasoro}
Let $V$ be a representation of the Virasoro algebra.
\begin{enumerate}
\item $V$ has \emph{central charge} $c\in\C$ if $K=c\cdot\id_V$.
\item We call $V$ \emph{positive-energy} if $L_0$ acts diagonalisably on $V$, i.e.\ $V$ is a direct sum of $L_0$-eigenspaces $V=\bigoplus_{\lambda\in\C}V_\lambda$, and if the subalgebra generated by $\{L_n\,|\,n\in\Ns\}$ acts locally nilpotently, i.e.\ for all $v\in V$ there is an $N\in\Ns$ such that $L_{n_1}\ldots L_{n_k}v=0$ for all sequences $n_1,\ldots,n_k\in\Ns$ satisfying $n_1+\ldots+n_k>N$. (The second property is trivially satisfied if the $L_0$-grading on $V$ is bounded from below.)
\item We say that a bilinear or sesquilinear form $(\cdot,\cdot)$ on $V$ is \emph{Virasoro-invariant} if $(L_nv,w)=(v,L_{-n}w)$ for all $v,w\in V$ and all $n\in\Z$.
\end{enumerate}
\end{defi}
For the BRST quantisation, which associates a physical space with a Virasoro representation, we first introduce the \emph{bosonic ghost vertex operator superalgebra} $V_\mathrm{gh.}$ of central charge $-26$ (in the ``ghost sector''). It can be constructed as the vertex operator superalgebra associated with the integral lattice $\Z\sigma$ with $\langle\sigma,\sigma\rangle=1$ and the usual Virasoro vector shifted by $\frac{3}{2}\sigma(-2)1\otimes\ee_0$.

$V_\mathrm{gh.}$ is $\Z_{\geq-1}$-graded by $L_0$-weights, $\Z_2$-graded by parity (super grading) and $\Z$-graded by ghost number, the eigenvalue of the ghost number operator $U=\sigma(0)$, and all these gradings are compatible. (In fact, the parity is just given by the parity of the ghost number.) $V_\mathrm{gh.}$ is generated by $b=1\otimes\ee_{-\sigma}$ and $c=1\otimes\ee_{\sigma}$, which have $L_0$-weights $2$ and $-1$, odd parity and ghost numbers $-1$ and $1$, respectively.

Given a Virasoro representation $M$ (in the ``matter sector'') of central charge $c$ we consider the tensor-product Virasoro module $W=M\otimes V_\mathrm{gh.}$, which is of central charge $c-26$. It is equipped with a tensor-product weight grading. The ghost (and parity) grading are extended trivially to the tensor product. Then one defines a BRST current $j^\mathrm{BRST}\in W$ and the BRST operator $Q:=j_0^\mathrm{BRST}$ as its zero mode. Explicitly,
\begin{align*}
Q&=\sum_{n\in\Z}L^M_n\otimes c_{-n-2}-\frac{1}{2}\id_M\otimes\sum_{m,n\in\Z}(m-n):c_{-m-2}c_{-n-2}b_{m+n+1}:
\end{align*}
where the normal ordering means that the annihilation operators $b_n$, $c_n$ for $n\geq0$ are moved to the right of the creation operators $b_n$, $c_n$ for $n\leq-1$, keeping track of minus signs since these are fermionic operators.
\begin{prop}
Let $M$ be a positive-energy representation of the Virasoro algebra of central charge $c$. The BRST operator $Q$ on $W=M\otimes V_\mathrm{gh.}$ fulfils:
\begin{enumerate}
\item $[U,Q]=Q$, i.e.\ $Q$ raises the ghost number by 1.
\item $[U,L_0]=0$, i.e.\ the ghost-number and $L_0$-grading are compatible.
\item $\{Q,b_{n+1}\}=L_n$ for all $n\in\Z$.
\item $Q^2=0$ if and only if $c=26$.
\end{enumerate}
Moreover, if $c=26$, then:
\begin{enumerate}[resume]
\item $[Q,L_n]=0$ for all $n\in\Z$.
\end{enumerate}
\end{prop}
\begin{proof}
These claims are readily checked. They are listed in \cite{Zuc89}, Section~4.
\end{proof}
We use the modern definition of $Q$ corresponding to the integral ghost grading described above rather than the version of $Q$ corresponding to the ghost grading shifted by $-3/2$, which is used in older texts.

If $c=26$, then the BRST operator $Q$ with ghost number $1$ satisfies $Q^2=0$, i.e.\ $\im(Q)\subseteq\ker(Q)$, and therefore defines a cochain complex of vector spaces, the \emph{BRST complex}
\begin{equation*}
\ldots\stackrel{Q}{\longrightarrow}W^{p-1}\stackrel{Q}{\longrightarrow}W^p\stackrel{Q}{\longrightarrow}W^{p+1}\stackrel{Q}{\longrightarrow}\ldots,
\end{equation*}
where $p$ denotes the ghost number. The complex is graded by $L_0$-weights because $[Q,L_0]=0$. Since $\{Q,b_1\}=L_0$, the corresponding cohomological spaces $H^p_\mathrm{BRST}(M)=(W^p\cap\ker(Q))/(W^p\cap\im(Q))$ are supported only in $L_0$-weight 0, which means we that can redefine the BRST complex to be
\begin{equation*}
\ldots\stackrel{Q}{\longrightarrow}W_0^{p-1}\stackrel{Q}{\longrightarrow}W_0^p\stackrel{Q}{\longrightarrow}W_0^{p+1}\stackrel{Q}{\longrightarrow}\ldots
\end{equation*}
without changing the cohomological spaces.

We can now define the BRST quantisation:
\begin{defi}\label{defi:brst}
Let $M$ be a positive-energy representation of the Virasoro algebra of central charge $26$. Then we define the \emph{physical space} to be $H^1_\mathrm{BRST}(M)$.
\end{defi}
Note that, in contrast to some of the cited literature, we use the term physical space irrespective of whether $H^1_\mathrm{BRST}(M)$ naturally admits a positive-definite Hermitian sesquilinear form or not (see also Remark~\ref{rem:noghost} below).
\begin{rem}
There is also a different quantisation procedure sometimes called \emph{old covariant quantisation}, which was used in \cite{Bor92}, for example. One can show, however, that given a positive-energy representation of the Virasoro algebra of central charge $26$ with a Virasoro-invariant bilinear form, the corresponding physical spaces are naturally isomorphic (see, for example, Lemma~3.3.6 in \cite{Car12b} and the references cited therein).
\end{rem}

The equation $\{Q,b_1\}=L_0$ also permits us to restrict $Q$ to $C:=W_0\cap\ker(b_1)$, the weight-zero vectors in the kernel of $b_1$, which defines the \emph{relative BRST subcomplex}
\begin{equation*}
\ldots\stackrel{Q}{\longrightarrow}C^{p-1}\stackrel{Q}{\longrightarrow}C^p\stackrel{Q}{\longrightarrow}C^{p+1}\stackrel{Q}{\longrightarrow}\ldots
\end{equation*}
with corresponding cohomological spaces $H^p_\mathrm{rel.}(M)=(C\cap\ker(Q))/(C\cap\im(Q))$. We note that the inclusion map $C^p\hookrightarrow W_0^p$ induces an injective map $H^p_\mathrm{rel.}(M)\to H^p_\mathrm{BRST}(M)$.

There is a short exact sequence of cochain complexes
\begin{equation*}
0\to C^\bullet\hookrightarrow W^\bullet_0\stackrel{\psi}{\longrightarrow}C^{\bullet-1}\to 0
\end{equation*}
with $\psi\colon W^p_0\to C^{p-1}$, $w\mapsto(-1)^{|w|}b_1w$. Then the zig-zag lemma entails a long exact sequence
\begin{equation*}
\ldots\to H_\mathrm{rel.}^p\to H_\mathrm{BRST}^p\to H_\mathrm{rel.}^{p-1}\to H_\mathrm{rel.}^{p+1}\to H_\mathrm{BRST}^{p+1}\to H_\mathrm{rel.}^{p}\to H_\mathrm{rel.}^{p+2}\to\ldots
\end{equation*}
In Section~\ref{sec:noghost} we shall study situations in which this sequence collapses.


\subsection{Lie Algebra and Invariant Bilinear Form}
We describe the case when the Virasoro representation $M$ in the matter sector is a conformal vertex algebra. Then $H_\mathrm{BRST}^1(M)$ and $H_\mathrm{rel.}^1(M)$ inherit Lie algebra structures. More precisely:
\begin{prop}\label{prop:valie}
Let $M$ be a conformal vertex algebra of central charge 26, which is positive-energy as Virasoro representation. Then the bracket $[u,v]=(b_0u)_0v$ for all $u,v\in W^1$ is well-defined on $H_\mathrm{BRST}^1(M)$ and endows it with the structure of a Lie algebra.

Moreover, the bracket restricts to $\ker(b_1)$ and also defines a Lie algebra structure on $H_\mathrm{rel.}^1(M)$.
\end{prop}
\begin{proof}
The first claim is stated in \cite{LZ93}, Theorem~2.2, and the second assertion follows from Lemma~2.1 in \cite{LZ93}.
\end{proof}
We note that if a group $G$ acts on $M$ by automorphisms of conformal vertex algebras, then $G$ induces an action on $H_\mathrm{BRST}^1(M)$ by Lie algebra automorphisms.

Let us additionally assume that the conformal vertex algebra $M$ carries a non-degenerate, invariant bilinear form $(\cdot,\cdot)$, which is necessarily symmetric \cite{Li94} and Virasoro-invariant. We show that this induces a non-degenerate, invariant bilinear form on the Lie algebra $H_\mathrm{rel.}^1(M)$ as well.

\begin{prop}\label{prop:Mformhrel}
Let $M$ be a positive-energy Virasoro representation of central charge $26$. Assume that $M$ is a conformal vertex algebra that carries a non-degenerate, invariant bilinear form $(\cdot,\cdot)_M$. Then $(\cdot,\cdot)_M$ induces a non-degenerate, symmetric, invariant bilinear form on the Lie algebra $H_\mathrm{rel.}^1(M)$.
\end{prop}
\begin{proof}
Since $M$ is positive-energy, in particular $L_1$ acts locally nilpotently on $M$. In this case there is still a nice theory of invariant bilinear forms on $M$ \cite{Sch97,Sch98}, similar to the theory for \voa{}s developed in \cite{Li94}. We shall also need to consider $\Z$-graded conformal vertex superalgebras, for which the theory is described in \cite{Sch00,Sch04}.

The proof closely follows the arguments made in \cite{Sch04}, Section~4, and \cite{Sch00}, Section~5. Note that $(V_\mathrm{gh.})_0/L_1^\mathrm{gh.}(V_\mathrm{gh.})_1$ is one-dimensional. Let $(\cdot,\cdot)_\mathrm{gh.}$ be the unique invariant bilinear form on the ghost vertex superalgebra $V_\mathrm{gh.}$. For definiteness we normalise it such that $(\vac,1\otimes\ee_{3\sigma})_\mathrm{gh.}=1$. Then $(\cdot,\cdot)_\mathrm{gh.}$ is super-symmetric (with respect to the $\Z_2$-grading), non-degenerate, vanishes on $\ker(b_1)$ and pairs spaces $(V_\mathrm{gh.})_n^p$ and $(V_\mathrm{gh.})_m^q$ non-trivially only if $m=n$ and $p+q=3$. Also, note that the following adjoint relations hold: $b_n^*=b_{2-n}$ and $c_n^*=-c_{-4-n}$ for all $n\in\Z$.

On $W=M\otimes V_\mathrm{gh.}$ we consider the tensor-product bilinear form $(\cdot,\cdot)_W$, which is non-degenerate, super-symmetric, invariant and vanishes on $\ker(b_1)$. Moreover, $Q^*=-Q$.

We then define a bilinear form on $B:=W\cap\ker(b_1)$ by setting $(u,v)_B:=(c_{-2}u,v)_W$ for $u,v\in B$. It is non-degenerate, super-antisymmetric and pairs $B_n^p$ and $B_m^q$ non-trivially only if $m=n$ and $p+q=2$.

Then $(u,v)_B$ can be restricted to $C=B_0=W_0\cap\ker(b_1)$ and the corresponding bilinear form is again non-degenerate. Moreover, $(Qu,v)_C=-(-1)^{|u|}(u,Qv)_C$ for $u,v\in C$ with $u$ being $\Z_2$-homogeneous.

The last relation, together with $Q^2=0$, entails that $(\cdot,\cdot)_C$ induces a well-defined bilinear form $(\cdot,\cdot)_{H_\mathrm{rel.}}$ on $H_\mathrm{rel.}=(\ker(Q)\cap C)/(\im(Q)\cap C)$. This form is non-degenerate, super-antisymmetric and pairs $H_\mathrm{rel.}^p$ and $H_\mathrm{rel.}^q$ non-trivially only if $p+q=2$.

Finally, $(\cdot,\cdot)_{H_\mathrm{rel.}}$ can be restricted to the Lie algebra $H_\mathrm{rel.}^1$ and the resulting bilinear form is non-degenerate, symmetric and invariant.
\end{proof}


\subsection{Vanishing Theorem}\label{sec:noghost}
In the following we shall specialise to the case where the Virasoro representation $M$ in the matter sector carries a representation of the Heisenberg (or free-boson) \voa{} $\pi_0^{(k-1,1)}$ of some rank $2\leq k\leq 26$ and Lorentzian signature.

The following vanishing theorem, which uses the full power Feigin's semi-infinite cohomology theory \cite{Fei84}, asserts the vanishing of almost all cohomological spaces associated with the relative BRST complex.
\begin{prop}[Vanishing Theorem, \cite{Zuc89}, Theorem~4.9, \cite{Fei84}]\label{prop:vanish}
Let $2\leq k\leq 26$ and $V$ be a positive-energy Virasoro representation of central charge $26-k$ carrying a non-degenerate, Virasoro-invariant Hermitian sesquilinear form. Let $\alpha\in\R^{(k-1,1)}\otimes_\R\C$ with $\alpha\neq 0$. Then
\begin{equation*}
H_\mathrm{rel.}^p(V\otimes\pi_\alpha^{(k-1,1)})=\{0\}
\end{equation*}
for all $p\neq 1$.
\end{prop}
Of course, for this result the \voa{} module structure of $\pi_\alpha^{(k-1,1)}$ is irrelevant. Only the structure of $M=V\otimes\pi_\alpha^{(k-1,1)}$ as a Virasoro module with a Virasoro-invariant Hermitian sesquilinear form matters.

We remark that a vanishing theorem for $M=\pi_\alpha^{(r,s)}$ for $r,s\in\Ns$ with $r+s=26$ was stated in Theorem~2.7 of \cite{FGZ86}.

The vanishing of the relative cohomological spaces for $\alpha\neq 0$ lets collapse the above long exact sequence so that for $\alpha\neq 0$
\begin{equation*}
H_\mathrm{BRST}^1(V\otimes\pi_\alpha^{(k-1,1)})\cong H_\mathrm{rel.}^1(V\otimes\pi_\alpha^{(k-1,1)})\cong H_\mathrm{BRST}^2(V\otimes\pi_\alpha^{(k-1,1)})
\end{equation*}
and
\begin{equation*}
H_\mathrm{BRST}^p(V\otimes\pi_\alpha^{(k-1,1)})=\{0\}
\end{equation*}
for all $p\neq 1,2$.
\begin{rem}\label{rem:noghost}
Like in the proof of Proposition~\ref{prop:Mformhrel}, the non-degenerate, Hermitian sesquilinear forms on $V$ and $\pi_\alpha^{(k-1,1)}$ induce a non-degenerate, Hermitian sesquilinear form on the physical space $H_\mathrm{BRST}^1(V\otimes\pi_\alpha^{(k-1,1)})\cong H_\mathrm{rel.}^1(V\otimes\pi_\alpha^{(k-1,1)})$ for $\alpha\neq 0$. If the form on $V$ is positive-definite, then so is the one on the physical space \cite{Zuc89}. This is referred to as \emph{no-ghost theorem}.
\end{rem}

The vanishing theorem together with the fact that the Euler-Poincaré characteristic of the relative BRST complex, if well-defined, is the same as the one of the corresponding cohomological spaces implies:
\begin{prop}[\cite{Zuc89}, Theorem~4.9]\label{prop:vanish2}
Let $2\leq k\leq 26$ and $V$ be a positive-energy Virasoro representation of central charge $26-k$ carrying a non-degenerate, Virasoro-invariant, Hermitian sesquilinear form. Assume that the $L_0$-grading of $V$ is bounded from below and that the $L_0$-eigenspaces of $V$ are finite-dimensional. Let $\alpha\in\R^{(k-1,1)}\otimes_\R\C$ with $\alpha\neq 0$. Then
\begin{equation*}
H_\mathrm{BRST}^1(V\otimes\pi_\alpha^{(k-1,1)})\cong H_\mathrm{rel.}^1(V\otimes\pi_\alpha^{(k-1,1)})\cong(V\otimes\pi_0^{(k-2,0)})_{1-\langle\alpha,\alpha\rangle/2}.
\end{equation*}
\end{prop}
(Note that $1/\Delta(q)$ should be replaced by $1/q$ in item (b) of Theorem~4.9 in \cite{Zuc89}.)

The case of $\alpha=0$ is not covered by the vanishing theorem but a direct calculation yields:
\begin{prop}[\cite{Zuc89}, Theorem~4.9]\label{prop:vanish2_0}
Let $2\leq k\leq 26$ and $V$ be a positive-energy Virasoro representation of central charge $26-k$ carrying a non-degenerate, Virasoro-invariant, Hermitian sesquilinear form. Assume that
\begin{enumerate*}
\item\label{item:cond1} the $L_0$-spectrum of $V$ is non-negative,
\item\label{item:cond2} $L_{-1}V_0=\{0\}$.
\end{enumerate*}
Then
\begin{equation*}
H_\mathrm{BRST}^1(V\otimes\pi_0^{(k-1,1)})\cong H_\mathrm{rel.}^1(V\otimes\pi_0^{(k-1,1)})\cong(V\otimes\pi_0^{(k-1,1)})_1.
\end{equation*}
\end{prop}
\begin{proof}
Assuming that $M$ is a positive-energy Virasoro representation with non-negative $L_0$-spectrum, one computes $H_\mathrm{BRST}^1(M)=((\ker(L_1)\cap M_1)/L_{-1}M_0)\otimes\C c=H_\mathrm{rel.}^1(M)$. Inserting $M=V\otimes\pi_0^{(k-1,1)}$ yields
\begin{equation*}
\frac{\ker(L_1)\cap V_1}{L_{-1}V_0}\otimes\C\vac\otimes\C c\oplus V_0\otimes(\pi_0^{(k-1,1)})_1\otimes \C c
\end{equation*}
Using \eqref{item:cond2}, which also implies $L_1V_1=\{0\}$ because of the non-degenerate, Virasoro-invariant, Hermitian sesquilinear form, this proves the assertion.
\end{proof}

If the character of $V$ is well-defined, the following is immediate with knowledge of the character of the Heisenberg \voa{} (see Section~\ref{sec:lattice}):
\begin{cor}
Let $2\leq k\leq 26$ and $V$ as in the above proposition. Additionally assume that all the $L_0$-eigenspaces are finite-dimensional. Then the dimension of the physical space is
\begin{equation*}
\dim(H_\mathrm{BRST}^1(V\otimes\pi_\alpha^{(k-1,1)}))=\left[\ch_{V}(q)/\eta(q)^{k-2}\right](-\langle\alpha,\alpha\rangle/2)+2\delta_{\alpha,0}\dim(V_0).
\end{equation*}
for all $\alpha\in\R^{(k-1,1)}\otimes_\R\C$.
\end{cor}


\subsection{Natural Construction of Ten \BKMA{}s}\label{sec:tenbkmas}
Finally, we apply the BRST quantisation to the ten conformal vertex algebras $M_{\phi_\nu}$ from Section~\ref{sec:tenM}.

First, we must check that the assumptions are satisfied.
\begin{lem}\label{lem:lieass}
Let $\nu$ be of square-free order in $M_{23}$. Then $M_{\phi_\nu}$ is a positive-energy Virasoro representation of central charge $26$.
\end{lem}
\begin{proof}
By definition, $M_{\phi_\nu}$ decomposes as
\begin{equation*}
M_{\phi_\nu}=\bigoplus_{\alpha\in K'}V_\Lambda^{\phi_\nu}(\varphi(\alpha+K))\otimes\pi^{(1,1)}_\alpha.
\end{equation*}
Clearly, $L_0$ acts diagonalisably on $M_{\phi_\nu}$ with central charge $26$. Moreover, the $L_0$-grading on the irreducible $V_\Lambda^{\phi_\nu}$-modules and on the Heisenberg modules $\pi^{(1,1)}_\alpha$, $\alpha\in K'$, is bounded from below. Hence, they are positive energy, which (in contrast to the boundedness from below) carries over to $M_{\phi_\nu}$.
\end{proof}
This allows us to apply the BRST quantisation in Definition~\ref{defi:brst} to $M_{\phi_\nu}$. We define
\begin{equation*}
\g^{\phi_\nu}:=H^1_\mathrm{BRST}(M_{\phi_\nu}),
\end{equation*}
which is a Lie algebra:
\begin{prop}
Let $\nu$ be of square-free order $m$ in $M_{23}$. Then the physical space $\g^{\phi_\nu}$ is an $L'$-graded Lie algebra, i.e.\
\begin{equation*}
\g^{\phi_\nu}=\bigoplus_{\alpha\in L'}\g^{\phi_\nu}(\alpha)\quad\text{and}\quad
[\g^{\phi_\nu}(\alpha),\g^{\phi_\nu}(\beta)]\subseteq \g^{\phi_\nu}(\alpha+\beta)
\end{equation*}
for all $\alpha,\beta\in L'$ where $\g^{\phi_\nu}(\alpha)=H^1_\mathrm{BRST}(M_{\phi_\nu}(\alpha))$ and $L=\Lambda^\nu\oplus \II_{1,1}(m)$.
\end{prop}
\begin{proof}
The Lie algebra claim follows from Lemma~\ref{lem:lieass} and Proposition~\ref{prop:valie}.

For the grading we recall that the conformal vertex algebra $M_{\phi_\nu}$ is graded by the dual lattice $L'$, i.e.\
\begin{equation*}
M_{\phi_\nu}=\bigoplus_{\alpha\in L'}M_{\phi_\nu}(\alpha)=\bigoplus_{\alpha\in L'}V_{\Lambda_\nu}^{\hat\nu}(\chi(\alpha+L))\otimes\pi^{(k-1,1)}_\alpha.
\end{equation*}
We note that the $L'$-grading is compatible with the $L_0$-grading on $M_{\phi_\nu}$. In fact, all the Virasoro modes $L_n$, $n\in\Z$, on $M_{\phi_\nu}$ preserve the $L'$-grading, and hence so does~$Q$. We conclude that the BRST quantisation preserves the $L'$-grading, which shows the direct-sum decomposition of $\g^{\phi_\nu}$. Since $M_{\phi_\nu}$ is $L'$-graded as vertex algebra, $\g^{\phi_\nu}$ is $L'$-graded as Lie algebra.
\end{proof}

The $L'$-decomposition of the Lie algebra $\g^{\phi_\nu}$ allows us to apply the vanishing theorem or its corollary, Proposition~\ref{prop:vanish2}. Again, we first have to check that the assumptions are satisfied:
\begin{lem}\label{lem:vanishass}
Let $\nu$ be of square-free order $m$ in $M_{23}$. Then $V_{\Lambda_\nu}^{\hat\nu}(\alpha+\Lambda_\nu,i,j)$ admits a non-degenerate, Virasoro-invariant Hermitian sesquilinear form and satisfies items \eqref{item:cond1} and \eqref{item:cond2} in Proposition~\ref{prop:vanish2_0} for all $\alpha+\Lambda_\nu\in(\Lambda_\nu)'/\Lambda_\nu$ and $i,j\in\Z_m$.
\end{lem}
\begin{proof}
That the $L_0$-spectrum of all the irreducible $V_{\Lambda_\nu}^{\hat\nu}$-modules is non-negative follows from the corresponding fact for the irreducible $\hat\nu^i$-twisted $V_{\Lambda_\nu}$-modules for $i\in\Z_m$. Their conformal weights are described in \cite{DL96} and always non-negative. This shows \eqref{item:cond1}. In fact, $V_{\Lambda_\nu}^{\hat\nu}$ satisfies the positivity condition, i.e.\ the conformal weight of any irreducible $V_{\Lambda_\nu}^{\hat\nu}$-module is positive except for that of $V_{\Lambda_\nu}^{\hat\nu}$ itself. Since $V_{\Lambda_\nu}^{\hat\nu}$ is of CFT-type and $L_{-1}\vac=0$, this shows \eqref{item:cond2}.

If $m=1$, $V_{\Lambda_\nu}^{\hat\nu}$ is the trivial \voa{}. For the remaining cases the central charge is $26-k=8,12,\ldots,22$. Similar to Lemma~3.1.2 in \cite{Car12b} (see also \cite{KRR13}) one can show that for all $\alpha+\Lambda_\nu\in(\Lambda_\nu)'/\Lambda_\nu$ and $i,j\in\Z_m$, $V_{\Lambda_\nu}^{\hat\nu}(\alpha+\Lambda_\nu,i,j)$ admits a non-degenerate, Virasoro-invariant, Hermitian sesquilinear form.
\end{proof}
The lemma permits us to compute $\g^{\phi_\nu}(\alpha)=H^1_\mathrm{BRST}(M_{\phi_\nu}(\alpha))\cong H^1_\mathrm{rel.}(M_{\phi_\nu}(\alpha))$ using Propositions \ref{prop:vanish2} and \ref{prop:vanish2_0}:
\begin{prop}\label{prop:tendims}
Let $\nu$ be of square-free order in $M_{23}$. Then the $L'$-graded Lie algebra $\g^{\phi_\nu}$ satisfies
\begin{equation*}
\g^{\phi_\nu}(\alpha)\cong\begin{cases}
(V_{\Lambda_\nu}^{\hat\nu}(\chi(\alpha+L))\otimes\pi_0^{(k-2,0)})_{1-\langle\alpha,\alpha\rangle/2}&\text{if }\alpha\neq0,\\
(V_{\Lambda_\nu}^{\hat\nu}\otimes\pi_0^{(k-1,1)})_1&\text{if }\alpha=0
\end{cases}
\end{equation*}
for all $\alpha\in L'$. Moreover,
\begin{equation*}
\dim(\g^{\phi_\nu}(\alpha))=\left[\ch_{V_{\Lambda_\nu}^{\hat\nu}(\chi(\alpha+L))}(q)/\eta(q)^{k-2}\right](-\langle\alpha,\alpha\rangle/2)
\end{equation*}
for all $\alpha\in L'\setminus\{0\}$ and
\begin{equation*}
\dim(\g^{\phi_\nu}(0))=k=\rk(L).
\end{equation*}
\end{prop}
\begin{proof}
The first two claims are immediate from Lemma~\ref{lem:vanishass} and Proposition~\ref{prop:vanish2}. The last statement follows since $V_{\Lambda_\nu}^{\hat\nu}$ is of CFT-type and satisfies $(V_{\Lambda_\nu}^{\hat\nu})_1=\{0\}$.
\end{proof}

By the above proposition, the dimensions of the graded components of $\g^{\phi_\nu}$ are Fourier coefficients exactly of the vector-valued modular form $F$ introduced in Section~\ref{sec:chars} (see Proposition~\ref{prop:F1eqF2}) and lifting to the automorphic product $\Psi_{\phi_\nu}$ (see Section~\ref{sec:twisting}). Hence:
\begin{cor}\label{cor:dimform}
Let $\nu$ be of square-free order $m$ in $M_{23}$. Then
\begin{equation*}
\dim(\g^{\phi_\nu}(\alpha))=\left[F_{\alpha+L}\right](-\langle\alpha,\alpha\rangle/2)=\sum_{d\mid m}\delta_{\alpha\in L'\cap \frac{1}{d}L}[1/\eta_\nu](-d\langle\alpha,\alpha\rangle/2)
\end{equation*}
for all $\alpha\in L'\setminus\{0\}$.
\end{cor}
\begin{proof}
Proposition~\ref{prop:simplelift} implies that
\begin{align*}
\left[F_{\alpha+L}\right](-\langle\alpha,\alpha\rangle/2)&=\sum_{d\mid m}\delta_{\alpha\in L'\cap \frac{1}{d}L}\,\left[g_{d,j_{\alpha+L,d}}\right](-\langle\alpha,\alpha\rangle/2)\\
&=\sum_{d\mid m}\delta_{\alpha\in L'\cap \frac{1}{d}L}[1/\eta_\nu](-d\langle\alpha,\alpha\rangle/2)
\end{align*}
by definition of the $g_{d,j}(\tau)$ in terms of $1/\eta_\nu(\tau/d)$.
\end{proof}

Because $\g^{\phi_\nu}=H^1_\mathrm{BRST}(M_{\phi_\nu})=H^1_\mathrm{rel.}(M_{\phi_\nu})$, we can use Proposition~\ref{prop:Mformhrel} to define a non-degenerate, symmetric, invariant bilinear form on $\g^{\phi_\nu}$.
\begin{lem}
Let $\nu$ be of square-free order in $M_{23}$. Then the conformal vertex algebra $M_{\phi_\nu}$ admits a non-degenerate, symmetric, invariant bilinear form $(\cdot,\cdot)_{M_{\phi_\nu}}$, which is unique up to a non-zero scalar.
\end{lem}
\begin{proof}
The space of symmetric, invariant bilinear forms on $M_{\phi_\nu}$ is isomorphic to the dual space of $(M_{\phi_\nu})_0/L_1(M_{\phi_\nu})_1$ since $L_1$ acts locally nilpotently on $M_{\phi_\nu}$ \cite{Sch97,Sch98}. However, instead of studying such forms on $M_{\phi_\nu}$ directly, we shall first consider the \voa{} $M_{\phi_\nu}(0)$ and then extend the result to $M_{\phi_\nu}$.

The space $M_{\phi_\nu}(0)_0/L_1M_{\phi_\nu}(0)_1$ is one-dimensional. Let $(\cdot,\cdot)_{M_{\phi_\nu}(0)}$ be the up to non-zero scalar unique non-degenerate, symmetric, invariant bilinear form on the simple, self-contragredient \voa{} $M_{\phi_\nu}(0)$. It is related to the contragredient pairing by $(u,v)_{M_{\phi_\nu}(0)}=\langle\phi_0(u),v\rangle$ for $u,v\in M_{\phi_\nu}(0)$ where $\phi_0\colon M_{\phi_\nu}(0)\to M_{\phi_\nu}(0)'$ is an isomorphism of $M_{\phi_\nu}(0)$-modules, again unique up to a non-zero scalar.

Any $M_{\phi_\nu}$-invariant bilinear form on $M_{\phi_\nu}$ can only pair the irreducible $M_{\phi_\nu}(0)$-module $M_{\phi_\nu}(\alpha)=V_{\Lambda_\nu}^{\hat\nu}(\chi(\alpha+L))\otimes\pi^{(k-1,1)}_\alpha$ non-trivially with its contragredient module $M_{\phi_\nu}(\alpha)'\cong M_{\phi_\nu}(-\alpha)=V_{\Lambda_\nu}^{\hat\nu}(\chi(-\alpha+L))\otimes\pi^{(k-1,1)}_{-\alpha}$, and such a form is in particular $M_{\phi_\nu}(0)$-invariant.

On the other hand, $(\cdot,\cdot)_{M_{\phi_\nu}(0)}$ and the contragredient pairings with choices of $M_{\phi_\nu}(0)$-module isomorphisms $\phi_\alpha\colon M_{\phi_\nu}(-\alpha)\to(M_{\phi_\nu}(\alpha))'$ for $\alpha\neq 0$ define a non-degenerate, symmetric, $M_{\phi_\nu}(0)$-invariant bilinear form $(\cdot,\cdot)_{M_{\phi_\nu}}$ on $M_{\phi_\nu}$.

Proper normalisation with respect to the normalisation of $(\cdot,\cdot)_{M_{\phi_\nu}(0)}$ (cf.\ Proposition~3.1.8 in \cite{Car12b}) makes the form $(\cdot,\cdot)_{M_{\phi_\nu}}$ $M_{\phi_\nu}$-invariant.
\end{proof}

For definiteness we normalise $(\cdot,\cdot)_{M_{\phi_\nu}}$ such that $(\vac,\vac)_{M_{\phi_\nu}}=1$. Proposition~\ref{prop:Mformhrel} implies:
\begin{prop}\label{prop:lainvbil}
Let $\nu$ be of square-free order in $M_{23}$. Then there is a non-degenerate, symmetric, invariant bilinear form $(\cdot,\cdot)_{\g^{\phi_\nu}}$ on $\g^{\phi_\nu}$.
\end{prop}

In the following we describe the zero-component $\mathcal{H}:=\g^{\phi_\nu}(0)$ of $\g^{\phi_\nu}$, which we shall later identify as a Cartan subalgebra of $\g^{\phi_\nu}$. It simplifies to
\begin{equation*}
\mathcal{H}\cong(V_{\Lambda_\nu}^{\hat\nu})_0\otimes(\pi_0^{(k-1,1)})_1=\C\vac\otimes\{h(-1)1\,|\,h\in\h\}\cong\h
\end{equation*}
with $\h:=L\otimes_\Z\C$ since $V_{\Lambda_\nu}^{\hat\nu}$ is of CFT-type and satisfies $(V_{\Lambda_\nu}^{\hat\nu})_1=\{0\}$.

Now recall that $(\cdot,\cdot)_{\g^{\phi_\nu}}$ is induced from the tensor product of the up to non-zero scalar unique invariant, bilinear forms $(\cdot,\cdot)_{M_{\phi_\nu}}$ on $M_{\phi_\nu}$ and $(\cdot,\cdot)_\mathrm{gh.}$ on $V_\mathrm{gh.}$ and that we chose normalisations for both. Moreover, recall that $\h$ comes equipped with a bilinear form $\langle\cdot,\cdot\rangle$ obtained as extension of the bilinear form $\langle\cdot,\cdot\rangle$ on the lattice $L$. Then the above isomorphism is even an isometry:
\begin{prop}\label{prop:Hiso}
Let $\nu$ be of square-free order in $M_{23}$. Then there is an isometry
\begin{equation*}
\left(\h,\langle\cdot,\cdot\rangle\right)\cong\left(\mathcal{H},(\cdot,\cdot)_{\g^{\phi_\nu}}\right)
\end{equation*}
induced by $h\mapsto\vac\otimes h(-1)1\otimes c\in W$ for all $h\in\h=L\otimes_\Z\C$. This isometry maps $L\otimes_\Z\R$, on which the bilinear form $\langle\cdot,\cdot\rangle$ is real-valued and of signature $(k-1,1)$, to a real subspace $\mathcal{H}_\R$ of $\mathcal{H}$ on which $(\cdot,\cdot)_{\g^{\phi_\nu}}$ is real-valued and of signature $(k-1,1)$.
\end{prop}
\begin{proof}
Cf. \cite{Sch04}, Section~4.2.
\end{proof}
We shall see that $\mathcal{H}=\g^{\phi_\nu}(0)$ is a Cartan subalgebra of $\g^{\phi_\nu}$. For this property it is essential that $(V_{\Lambda_\nu}^{\hat\nu})_1=\{0\}$.

In the following we prove that $\g^{\phi_\nu}$ is a \BKMa{} using Propositions \ref{prop:bkma} and \ref{prop:bor95thm2}.
\begin{lem}
Let $\nu$ be of square-free order in $M_{23}$. Then $\g^{\phi_\nu}$ satisfies items \eqref{enum:bkma1} to \eqref{enum:bkma4} in Proposition~\ref{prop:bkma}.
\end{lem}
\begin{proof}
Item~\eqref{enum:bkma1} is the statement of Proposition~\ref{prop:lainvbil}.

Recall that $\g^{\phi_\nu}$ is graded by $L'$. Then $\mathcal{H}=\g^{\phi_\nu}(0)$ is a Lie subalgebra of $\g^{\phi_\nu}$ and acts on $\g^{\phi_\nu}$ in the adjoint representation as
$[x,y]=\langle h,\alpha\rangle y$ for $x=\vac\otimes h(-1)1\otimes c\in \mathcal{H}$ and $y\in\g^{\phi_\nu}(\alpha)$, $\alpha\in L'$. This implies that $\mathcal{H}$ is self-centralising.

We abuse notation and write $h\in\h$ for the element $\vac\otimes h(-1)1\otimes c\in\mathcal{H}=\g^{\phi_\nu}(0)$, identifying $\mathcal{H}$ with $\h$. Since the bilinear form on $\mathcal{H}$ is non-degenerate, we can further identify $\mathcal{H}\cong\h$ with $\h^*$ via $\alpha(\cdot):=\langle\alpha,\cdot\rangle$ for $\alpha\in\h$. Then
\begin{equation*}
[h,x]=\alpha(h)x
\end{equation*}
for $h\in \mathcal{H}$ and $x\in\g^{\phi_\nu}(\alpha)$, i.e.\ $\g^{\phi_\nu}(\alpha)$ is the root space associated with $\alpha\in L'\setminus\{0\}$. The set of roots $\Phi\subseteq L'\setminus\{0\}$ are those $\alpha$ for which $\g^{\phi_\nu}(\alpha)\neq\{0\}$. Then $\g^{\phi_\nu}$ decomposes into the direct sum
\begin{equation*}
\g^{\phi_\nu}=\mathcal{H}\oplus\bigoplus_{\alpha\in\Phi}\g^{\phi_\nu}(\alpha)
\end{equation*}
with Cartan subalgebra $\mathcal{H}$ and root spaces $\g^{\phi_\nu}(\alpha)$, $\alpha\in\Phi$. Proposition~\ref{prop:tendims} states in particular that $\dim(\g^{\phi_\nu}(\alpha))<\infty$ for all $\alpha\in L'\setminus\{0\}$, i.e.\ the root spaces are finite-dimensional. This completes the proof of item~\eqref{enum:bkma2}.

Proposition~\ref{prop:Hiso} isometrically identifies $\mathcal{H}$ with $\h=L\otimes_\Z\C$, which has a natural real subspace $\mathcal{H}_\R:=L\otimes_\Z\R$, on which the bilinear form takes real values, and the roots, identified with elements of the lattice $L'$, lie in $\mathcal{H}_\R^*$. This shows item~\eqref{enum:bkma3}.

Under the identifications presented above the norm of a root $\alpha\in\Phi$ is exactly $\langle\alpha,\alpha\rangle/2$. From the explicit expression for $\g^{\phi_\nu}(\alpha)$ in Proposition~\ref{prop:tendims} we conclude that $\g^{\phi_\nu}(\alpha)=\{0\}$ if $\langle\alpha,\alpha\rangle/2>1$ since $V_{\Lambda_\nu}^{\hat\nu}$ satisfies the positivity condition. This proves \eqref{enum:bkma4}.
\end{proof}
The more difficult part of the proof that $\g^{\phi_\nu}$ is a \BKMa{} is to show that the conditions in Proposition~\ref{prop:bor95thm2} are satisfied. First, we need the following lemma:
\begin{lem}\label{lem:largeorbits}
Let $\nu$ be of square-free order $m$ in $M_{23}$ and $L=\Lambda^\nu\oplus\II_{1,1}(m)$. Then the orbits of the \fqs{} $L'/L$ under $\O(L'/L)$ are uniquely determined by the order and the value of the quadratic form of their elements.
\end{lem}
\begin{proof}
Proposition~5.1 in \cite{Sch15} implies that for a non-degenerate \fqs{} $D$ of square-free level, two elements of $D$ are in the same orbit under $\O(D)$ if and only if they have the same order and value of the quadratic form (see comment before Proposition~5.3 in \cite{Sch15}). Since $L'/L$ has level $m$, the assertion follows.
\end{proof}
\begin{lem}
Let $\nu$ be of square-free order in $M_{23}$. Then $\g^{\phi_\nu}$ satisfies the conditions in Proposition~\ref{prop:bor95thm2}, which implies that \eqref{enum:bkma5} and \eqref{enum:bkma6} in Proposition~\ref{prop:bkma} are satisfied.
\end{lem}
\begin{proof}
We want to show that the root spaces of $\g^{\phi_\nu}$ corresponding to positive multiples of the same norm-zero root commute. To this end we consider the \voa{} $V_{\Lambda_\nu}^{\hat\nu}\otimes V_{\Lambda^\nu}$ of central charge $24$. Its fusion group is the \fqs{}
\begin{equation*}
F:=(\Lambda_\nu)'/\Lambda_\nu\times(\Z_m\times\Z_m,Q_m)\times(\Lambda^\nu)'/\Lambda^\nu
\end{equation*}
by the results in Section~\ref{sec:orb}. Let $I\leq F$ be an isotropic subgroup of $F$ with $I^\bot=I$. Then, as explained in Section~\ref{sec:ext}, the direct sum of irreducible $V_{\Lambda_\nu}^{\hat\nu}\otimes V_{\Lambda^\nu}$-modules
\begin{equation*}
V_I=\!\!\!\!\bigoplus_{(\alpha+\Lambda_\nu,i,j,\beta+\Lambda^\nu)\in I}\!\!\!\!V_{\Lambda_\nu}^{\hat\nu}(\alpha+\Lambda_\nu,i,j)\otimes V_{\beta+\Lambda^\nu}
\end{equation*}
is a \strat{}, holomorphic \voa{} of central charge $24$. These \voa{}s have been studied extensively (see, for example, \cite{Sch93,DM04,EMS20a}). In particular, $\dim((V_I)_1)=24$ if and only if $V_I$ is isomorphic to the lattice \voa{} $V_\Lambda$ associated with the Leech lattice $\Lambda$ \cite{DM04b}.

It is well-known that the weight-one space $V_1$ of a \voa{} $V$ of CFT-type carries the structure of a Lie algebra via $[u,v]=u_0v$ for all $u,v\in V_1$. Now, if $V_I\cong V_\Lambda$, then the Lie algebra $(V_I)_1$ is abelian of dimension $24$.

After these preliminary considerations, let $\gamma\in L'\setminus\{0\}$ with $\langle\gamma,\gamma\rangle/2=0$. Without loss of generality we may assume that $\gamma+L$ has maximal order $m$ in $L'/L$, a group of exponent $m$. Then $\chi(\gamma+L)$ is an isotropic element in $(\Lambda_\nu)'/\Lambda_\nu\times(\Z_m\times\Z_m,Q_m)$ of order $m$ and $(\chi(\gamma+L),0+\Lambda^\nu)$ is an isotropic element in $F$ of order $m$.

By Lemma~\ref{lem:largeorbits} there exists an automorphism $\kappa$ of the \fqs{} $(\Lambda_\nu)'/\Lambda_\nu\times(\Z_m\times\Z_m,Q_m)$ such that $\chi(\gamma+L)=\kappa((0+\Lambda_\nu,0,1))$. Define
\begin{equation*}
I:=\{(\psi(\lambda+\Lambda^\nu),0,i,\lambda+\Lambda^\nu)\,|\,\lambda+\Lambda^\nu\in(\Lambda^\nu)'/\Lambda^\nu,i\in\Z_m\}\leq F,
\end{equation*}
which is isotropic, satisfies $I^\bot=I$ and contains $(\kappa^{-1}(\chi(\gamma+L)),0+\Lambda^\nu)$.

Now consider the holomorphic \voa{} $V_I$ of central charge $24$ associated to this particular choice of $I$. We shall show that $\dim((V_I)_1)=24$ and hence $V_I\cong V_\Lambda$ so that $(V_I)_1$ is abelian. In fact, because the characters of the irreducible $V_{\Lambda_\nu}^{\hat\nu}$-modules have the special property that they are invariant under the automorphisms of $(\Lambda_\nu)'/\Lambda_\nu\times(\Z_m\times\Z_m,Q_m)$ (see item~\eqref{item:rem2} of Remark~\ref{rem:vvmfF}), it follows that
\begin{equation*}
V_{\hat{\kappa}(I)}=\bigoplus_{(\alpha+\Lambda_\nu,i,j,\beta+\Lambda^\nu)\in I}V_{\Lambda_\nu}^{\hat\nu}(\kappa(\alpha+\Lambda_\nu,i,j))\otimes V_{\beta+\Lambda^\nu}
\end{equation*}
has the same character as $V_I$ where $\hat{\kappa}=(\kappa,\id)\in\O(F)$ and hence $\dim((V_{\hat{\kappa}(I)})_1)=\dim((V_I)_1)=24$. In particular, $(V_{\hat{\kappa}(I)})_1$ is abelian.

But $\hat{\kappa}(I)$ contains the subgroup $\langle(\chi(\gamma+L),0+\Lambda^\nu)\rangle$, and therefore the abelian Lie algebra $(V_{\hat{\kappa}(I)})_1$ contains
\begin{align*}
\left(V_{\Lambda_\nu}^{\hat\nu}(k\chi(\gamma+L))\otimes V_{\Lambda^\nu}\right)_1&\cong V_{\Lambda_\nu}^{\hat\nu}(k\chi(\gamma+L))_1=V_{\Lambda_\nu}^{\hat\nu}(\chi(k\gamma+L))_1\\
&\cong\g^{\phi_\nu}(k\gamma)
\end{align*}
for all $k\in\Ns$. One checks that the definitions of the Lie brackets on the left-hand and right-hand side of the equation coincide, which implies that
\begin{equation*}
[\g^{\phi_\nu}(k\gamma),\g^{\phi_\nu}(l\gamma)]=0
\end{equation*}
for all $k,l\in\Ns$. This proves the assertion.

It remains to show that the holomorphic \voa{} $V_I$ has a weight-one space of dimension $24$. By the definition of $V_I$ and Proposition~\ref{prop:F1eqF2}, the character of $V_I$ is
\begin{align*}
\ch_{V_I}(\tau)&=\!\!\!\!\sum_{\substack{\lambda+\Lambda^\nu\in(\Lambda^\nu)'/\Lambda^\nu\\i\in\Z_m}}\!\!\!\!\ch_{V_{\Lambda_\nu}^{\hat\nu}(\psi(\lambda+\Lambda^\nu),0,i)}(\tau)\ch_{V_{\lambda+\Lambda^\nu}}(\tau)\\
&=\!\!\!\!\sum_{\substack{\lambda+\Lambda^\nu\in(\Lambda^\nu)'/\Lambda^\nu\\i\in\Z_m}}\!\!\!\!F_{\chi^{-1}(\psi(\lambda+\Lambda^\nu),0,i)}(\tau)\vartheta_{\lambda+\Lambda^\nu}(\tau).
\end{align*}
To determine the constant term in the $q$-expansion of the above character we note that $\vartheta_{\lambda+\Lambda^\nu}(\tau)$ has no singular terms and a constant term only if $\lambda+\Lambda^\nu=0+\Lambda^\nu$. As described in item~\eqref{item:rem1} of Remark~\ref{rem:vvmfF}, the complete reflectivity of the vector-valued modular form $F$ means that singular terms in the $q$-expansion of $F$ appear exactly in the components $F_{\alpha+L}(\tau)$, $\alpha+L\in L'/L$, with $\langle\alpha,\alpha\rangle/2=1/d\pmod{1}$ and $d\cdot(\alpha+L)=0+L$ for $d\mid m$ and that in such a component the only singular term is $1\cdot q^{-1/d}$. Hence
\begin{align*}
\dim((V_I)_1)&=[\ch_{V_I}](0)=\sum_{i\in\Z_m}[F_{\chi^{-1}(0+\Lambda_\nu,0,i)}](0)[\vartheta_{0+\Lambda^\nu}](0)\\
&\quad+\sum_{d\mid m}\sum_{\substack{\lambda+\Lambda^\nu\in(\Lambda^\nu)'/\Lambda^\nu\\i\in\Z_m,d\cdot i=0\\\langle\lambda,\lambda\rangle/2=1/d\pmod{1}\\d\cdot(\lambda+\Lambda^\nu)=0+\Lambda^\nu}}\!\!\!\![F_{\chi^{-1}(\psi(\lambda+\Lambda^\nu),0,i)}](-1/d)[\vartheta_{\lambda+\Lambda^\nu}](1/d)\\
&=\sum_{i\in\Z_m}[F_{\chi^{-1}(0+\Lambda_\nu,0,i)}](0)+\sum_{d\mid m}d\!\!\!\!\sum_{\substack{\lambda+\Lambda^\nu\in(\Lambda^\nu)'/\Lambda^\nu\\\langle\lambda,\lambda\rangle/2=1/d\pmod{1}\\d\cdot(\lambda+\Lambda^\nu)=0+\Lambda^\nu}}\!\!\!\![\vartheta_{\lambda+\Lambda^\nu}](1/d).
\end{align*}
Studying the theta series of the cosets of $\Lambda^\nu$ we find the second term to vanish. For example, since $\Lambda^\nu$ has no vectors $\alpha$ of norm $\langle\alpha,\alpha\rangle/2=1$, the coefficient of the $q^1$-term in $\vartheta_{\Lambda^\nu}(\tau)$ vanishes. Then
\begin{align*}
\dim((V_I)_1)&=\sum_{i\in\Z_m}[F_{\chi^{-1}(0+\Lambda_\nu,0,i)}](0)=\sum_{i\in\Z_m}[\ch_{V_{\Lambda_\nu}^{\hat\nu}(0+\Lambda_\nu,0,i)}/\eta^{k-2}](0)\\
&=[\ch_{V_{\Lambda_\nu}}/\eta^{k-2}](0)=[\vartheta_{\Lambda_\nu}/\eta^{24}](0)=24
\end{align*}
since also $\Lambda_\nu$ has no vectors $\alpha$ of norm $\langle\alpha,\alpha\rangle/2=1$. This completes the proof.
\end{proof}
The two lemmata imply:
\begin{prop}
Let $\nu$ be of square-free order in $M_{23}$. Then $\g^{\phi_\nu}$ is a \BKMa{} with Cartan subalgebra $\mathcal{H}=\g^{\phi_\nu}(0)\cong L\otimes_\Z\C$.
\end{prop}

Finally, we shall prove that $\g^{\phi_\nu}=H^1_\mathrm{BRST}(M_{\phi_\nu})$ is isomorphic to the complexification of the real \BKMa{} $\g_{\phi_\nu}$ constructed by Borcherds \cite{Bor92} by twisting the denominator identity of the Fake Monster Lie algebra $\g$ (see Section~\ref{sec:twisting}).

To facilitate the discussion we rescale the rational lattice $L'=(\Lambda^\nu)'\oplus (\II_{1,1}(m))'$, by which $\g^{\phi_\nu}$ is graded, to an even and in particular integral lattice $\Delta$. Note that $(\II_{1,1}(m))'\cong\II_{1,1}(1/m)$ and, due to the special form of the ten automorphisms, $(\Lambda^\nu)'\cong \Lambda^\nu(1/m)$. Hence, rescaling the quadratic form on $L$ by $m$ we obtain the even lattice
\begin{equation*}
\Delta:=L'(m)\cong\Lambda^\nu\oplus \II_{1,1},
\end{equation*}

Then Corollary~\ref{cor:dimform} implies:
\begin{cor}\label{cor:rescaleddim}
Let $\nu$ be of square-free order $m$ in $M_{23}$. Then the \BKMa{} $\g^{\phi_\nu}$ is graded by the even lattice $\Delta=L'(m)=\Lambda^\nu\oplus \II_{1,1}$ of rank $k$ and level $m$ with the dimensions of the graded components given by
\begin{align*}
\dim(\g^{\phi_\nu}(\alpha))&=\sum_{d\mid m}\delta_{\alpha\in\Delta\cap \frac{m}{d}\Delta'}\left[\frac{1}{\eta_\nu}\right]\left(-\frac{d}{m}\frac{\langle\alpha,\alpha\rangle}{2}\right)\\
&=\sum_{d\mid m}\delta_{\alpha\in\Delta\cap d\Delta'}\left[\frac{1}{\eta_\nu}\right]\left(-\frac{1}{d}\frac{\langle\alpha,\alpha\rangle}{2}\right)
\end{align*}
for all $\alpha\in\Delta\setminus\{0\}$.
\end{cor}
Comparing with the equation for $\dim(\g_{\phi_\nu}(\alpha))$ in Section~\ref{sec:twisting} this immediately shows that the $\Delta$-graded components, i.e.\ the root spaces, of $\g^{\phi_\nu}$ and $\g_{\phi_\nu}$ have identical dimensions.

We now study the roots of $\g^{\phi_\nu}$. The real roots of $\g^{\phi_\nu}$, i.e.\ the vectors $\alpha\in\Delta$ with $\langle\alpha,\alpha\rangle>0$ and $\dim(\g^{\phi_\nu}(\alpha))>0$, can be easily read off from the dimension formula in Corollary~\ref{cor:dimform} or the rescaled version in Corollary~\ref{cor:rescaleddim}:
\begin{prop}\label{prop:realroots}
Let $\nu$ be of square-free order $m$ in $M_{23}$. Then the real roots of $\g^{\phi_\nu}$ are exactly the $\alpha\in\Delta\cap d\Delta'$ with $\langle\alpha,\alpha\rangle/2=d$ for $d\mid m$, and they all have root multiplicity $\dim(\g^{\phi_\nu}(\alpha))=1$. Moreover, the real roots of $\g^{\phi_\nu}$ are exactly the roots of the lattice $\Delta$.
\end{prop}
\begin{proof}
Let $\alpha\in\Delta$ such that $\langle\alpha,\alpha\rangle>0$. Then using that
\begin{equation*}
\frac{1}{\eta_\nu(\tau)}=\prod_{t\mid m}\eta(t\tau)^{-24/\sigma_1(m)}=\frac{1}{q}+\frac{24}{\sigma_1(m)}+\ldots
\end{equation*}
we obtain
\begin{equation*}
\dim(\g^{\phi_\nu}(\alpha))=\sum_{d\mid m}\delta_{\alpha\in\Delta\cap d\Delta'}\left[\frac{1}{\eta_\nu}\right]\left(-\frac{1}{d}\frac{\langle\alpha,\alpha\rangle}{2}\right)=\sum_{d\mid m}\delta_{\alpha\in\Delta\cap d\Delta'}\delta_{d,\langle\alpha,\alpha\rangle/2},
\end{equation*}
which proves the first claim. The second claim follows directly from Propositions 2.1 and 2.2 in \cite{Sch06}.
\end{proof}

The Weyl group $W\leq\Aut(\Delta)$ of $\g^{\phi_\nu}$ is defined as the group generated by the reflections through the hyperplanes orthogonal to the real roots of $\g^{\phi_\nu}$ and hence in this case it is the full reflection group of the lattice $\Delta$, i.e.\ the group generated by the reflections through the hyperplanes orthogonal to the roots of $\Delta$.

Therefore, a choice of simple roots of the reflection group of $\Delta$ gives a choice of real simple roots of $\g^{\phi_\nu}$.

A Weyl vector for $W$ is a vector $\rho\in\Delta\otimes_\Z\R$ such that a set of simple roots of $W$ is given by the roots $\alpha\in\Delta$ satisfying $\langle\alpha,\rho\rangle=-\langle\alpha,\alpha\rangle/2$ (see Corollary~2.4 in \cite{Bor88}).
\begin{prop}\label{prop:weylvector}
Let $\nu$ be of square-free order in $M_{23}$. Then there exists a primitive norm-zero vector $\rho\in\Delta$ that is a Weyl vector for the reflection group $W$ of $\Delta$.
\end{prop}
\begin{proof}
As remarked earlier, the even lattice $\Lambda^\nu$ has no roots. This allows us to apply Theorem~3.3 in \cite{Bor90b} to the Lorentzian lattice $\Delta=\Lambda^\nu\oplus \II_{1,1}$. It states that there is a norm-zero vector $\rho\in\Delta$ such that the simple roots of the reflection group $W$ of $\Delta$ are exactly the roots $\alpha$ of $\Delta$ such that $\langle\alpha,\rho\rangle$ is negative and divides $\langle\alpha,v\rangle$ for all vectors $v\in\Delta$. It is not difficult to show that for one of the ten automorphisms the vector $\rho$ is a Weyl vector and primitive.
\end{proof}
A possible choice of Weyl vector is given by $\rho=(0,\eta)\in\Lambda^\nu\oplus\II_{1,1}$ for any primitive norm-zero vector $\eta\in\II_{1,1}$ (cf.\ \cite{CKS07}, directly before Theorem~6.2). We fix such a choice of $\rho$, which also fixes a set of simple roots of $W$ and the fundamental Weyl chamber, i.e.\ the set of vectors in $\Delta\otimes_\Z\R$ with non-positive inner product with the simple roots. (For example, we may take $\rho=(0,0,1)$, like for $\g_{\phi_\nu}$ in Theorem~\ref{thm:twistedbkma}.) The Weyl vector $\rho$ lies in the fundamental Weyl chamber. We obtain:
\begin{prop}
Let $\nu$ be of square-free order $m$ in $M_{23}$. Then the real simple roots of $\g^{\phi_\nu}$ are the $\alpha\in\Delta\cap d\Delta'$ with $\langle\alpha,\alpha\rangle/2=d$ for $d\mid m$ and $\langle\rho,\alpha\rangle=-\langle\alpha,\alpha\rangle/2$. These are precisely the simple roots of the reflection group $W$ of $\Delta$.
\end{prop}
\begin{proof}
This follows immediately from Proposition~\ref{prop:weylvector} and the properties of a Weyl vector.
\end{proof}

We then determine the imaginary simple roots of $\g^{\phi_\nu}$.
\begin{prop}
Let $\nu$ be of square-free order $m$ in $M_{23}$. Then the positive multiples $n\rho$, $n\in\Ns$, of the Weyl vector $\rho$ are imaginary simple roots of $\g^{\phi_\nu}$ with multiplicity $24\sigma_0((m,n))/\sigma_1(m)$.
\end{prop}
\begin{proof}
The Weyl vector $\rho$ lies in the fundamental Weyl chamber. In fact, it has negative inner product with all real simple roots. By Proposition~2.1 in \cite{Bor88} we can choose imaginary simple roots lying in the fundamental Weyl chamber so that $\rho$ has non-negative inner product with all simple roots. In Lorentzian signature the inner product of two vectors of non-positive norm in the same cone is non-positive and zero only if both vectors are multiples of the same norm-zero vector. Therefore, if we write $n\rho$, $n\in\Ns$, as sum of simple roots with positive coefficients, the only simple roots appearing in this sum are positive multiples of $\rho$. Since the support of an imaginary root is connected, all the $n\rho$, $n\in\Ns$, are simple roots. By Corollary~\ref{cor:rescaleddim}, the multiplicities are
\begin{equation*}
\dim(\g^{\phi_\nu}(n\rho))=\sum_{d\mid m}\delta_{n\rho\in\Delta\cap d\Delta'}[1/\eta_\nu](0)=\frac{24}{\sigma_1(m)}\sum_{d\mid m}\delta_{n\rho\in\Delta\cap d\Delta'}
\end{equation*}
for $n\in\Ns$. Since the Weyl vector $\rho=(0,\eta)$ is primitive in $\Delta=\Lambda^\nu\oplus \II_{1,1}$, we obtain that $n\rho\in\Delta\cap d\Delta'$ if and only if $d\mid n$ and hence
\begin{equation*}
\dim(\g^{\phi_\nu}(n\rho))=\frac{24}{\sigma_1(m)}\sum_{d\mid m}\delta_{d\mid n}=\frac{24\sigma_0((m,n))}{\sigma_1(m)}
\end{equation*}
for $n\in\Ns$, which completes the proof.
\end{proof}

The following result shows that these are in fact all the imaginary simple roots. The argument uses that the denominator identity of $\g^{\phi_\nu}$ (see also Corollary~\ref{cor:denid}) is the automorphic product $\Psi_{\phi_\nu}$ from \cite{Sch04,Sch06}.
\begin{prop}\label{prop:main}
Let $\nu$ be of square-free order $m$ in $M_{23}$. Then a set of simple roots of $\g^{\phi_\nu}$ is as follows: the real simple roots of $\g^{\phi_\nu}$ are the $\alpha\in\Delta\cap d\Delta'$ with $\langle\alpha,\alpha\rangle/2=d$ for $d\mid m$ and $\langle\rho,\alpha\rangle=-\langle\alpha,\alpha\rangle/2$ with multiplicity $1$ and the imaginary simple roots are the positive multiples $n\rho$, $n\in\Ns$, of the Weyl vector $\rho$ with multiplicity $24\sigma_0((m,n))/\sigma_1(m)$.
\end{prop}
\begin{proof}
We consider the automorphic product $\Psi_{\phi_\nu}$ of singular weight obtained in \cite{Sch04} as Borcherds lift of the vector-valued modular form $F$ introduced in Section~\ref{sec:chars}. Its expansion at any cusp is given by
\begin{equation*}
\ee^\rho\prod_{d\mid m}\prod_{\alpha\in\Phi^+\cap d\Delta'}(1-\ee^\alpha)^{[1/\eta_\nu](-\langle\alpha,\alpha\rangle/2d)}=\sum_{w\in W}\det(w)w(\eta_\nu(\ee^\rho)).
\end{equation*}

Now, let $\mathfrak{k}$ be the \BKMa{} with root lattice $\Delta$, Cartan subalgebra $\Delta\otimes_\Z\C$ and simple roots as stated in the theorem. Then the above is the denominator identity of $\mathfrak{k}$, implying that $\mathfrak{k}$ and $\g^{\phi_\nu}$ have the same root multiplicities (cf.\ proof of Theorem~7.2 in \cite{Bor92}). The simple roots of a \BKMa{} (with given Cartan subalgebra and choice of fundamental Weyl chamber) are determined by its root multiplicities because of the denominator identity. Hence, $\mathfrak{k}$ and $\g^{\phi_\nu}$ have the same simple roots (and are therefore isomorphic).
\end{proof}

The following two results are immediate corollaries of (the proof of) Proposition~\ref{prop:main}.
\begin{cor}\label{cor:denid}
Let $\nu$ be of square-free order $m$ in $M_{23}$. Then the denominator identity of the \BKMa{} $\g^{\phi_\nu}$ is
\begin{equation*}
\ee^\rho\prod_{d\mid m}\prod_{\alpha\in\Phi^+\cap d\Delta'}(1-\ee^\alpha)^{[1/\eta_\nu](-\langle\alpha,\alpha\rangle/2d)}=\sum_{w\in W}\det(w)w(\eta_\nu(\ee^\rho))
\end{equation*}
with $\Delta=\Lambda^\nu\oplus\II_{1,1}$, Weyl vector $\rho=(0,0,1)$ and Weyl group $W$, which is the full reflection group of $\Delta$.
\end{cor}

Comparing with Theorem~\ref{thm:twistedbkma} we obtain the main result of this work:
\begin{thm}[Main Result]\label{thm:main}
Let $\nu$ be of square-free order in $M_{23}$. Then $\g^{\phi_\nu}=H^1_\mathrm{BRST}(M_{\phi_\nu})$ is isomorphic to the complexification of $\g_{\phi_\nu}$.
\end{thm}

With the above theorem we have found a uniform, natural construction of the \BKMa{}s obtained in \cite{Bor92} by twisting the denominator identity of the Fake Monster Lie algebra $\g$ by elements of square-free order in $M_{23}$. These are also the ten \BKMa{}s classified in \cite{Sch06} whose denominator identities are completely reflective automorphic products of singular weight.

Moreover, we showed that these denominator identities are Borcherds lifts of the vector-valued characters of the \voa{}s in the input of this natural construction.

The main results are summarised in the following diagram (cf.\ the diagram in the introduction):
\begin{equation*}
\begin{tikzcd}
\text{Vertex alg.}&&\text{BKMA}&&\text{Aut.\ prod.}\\
V_\Lambda\otimes V_{\II_{1,1}}\cong V_{\II_{25,1}}\arrow{dd}{\phi_\nu}\arrow{rr}{\text{quantise}}\arrow[bend right]{rrrr}{\text{lift of char.}}&&\text{FMA }\g\arrow[<->]{rr}{\text{den. id.}}&&\Psi\arrow{dd}{\phi_\nu}\\
\\
\bigoplus_{\gamma\in K'/K}V_\Lambda^{\phi_\nu}(\gamma+K)\otimes V_{\gamma+K}\arrow{rr}{\text{quantise}}\arrow[bend right]{rrrr}{\text{lift of char.}}&&\g^{\phi_\nu}\cong\g_{\phi_\nu}\arrow[<->]{rr}{\text{den.\ id.}}&&\Psi_{\phi_\nu}
\end{tikzcd}
\end{equation*}

While we gave the first systematic string-theoretic construction of a subfamily of Borcherds' twisted versions of the Fake Monster Lie algebra, the majority of these Borcherds-Kac-Moody (super)algebras have not yet been realised in natural constructions (see Problem~3 in \cite{Bor92}). However, with recent advancements in orbifold theory, it should be possible to make further strides in this direction.


\bibliographystyle{alpha_noseriescomma}
\bibliography{quellen}{}

\end{document}